\documentclass[11pt, epsfig]{article}
\usepackage{epsfig, amsmath, amssymb, amsthm, times}
\usepackage[titletoc]{appendix}
\usepackage{graphicx}

\newcommand{\R}{{\mathbb R}}

\newcommand{\N}{{\mathbb N}}
\newcommand{\Z}{{\mathbb Z}}

\def\1{{\mathbf 1}}

\DeclareMathOperator{\esssup}{ess\,sup}

\usepackage[mathcal]{eucal}

\newtheorem {thm}{Theorem}[section]
\newtheorem {lem}[thm]{Lemma}

\newtheorem {cor}[thm]{Corollary}

\theoremstyle{defintion}
\newtheorem {df}[thm]{Definition}

\theoremstyle{remark}
\newtheorem{rem}[thm]{Remark}

\theoremstyle{example}

\theoremstyle{assumption}

\def\E{\operatorname{\mathbb E}}
\def\P{\operatorname{\mathbb P}}
\def\T{{\mathbb T}}

\def\lbl{\label}
\def\be{\begin{equation}}
\def\ee{\end{equation}}
\def\p{\partial}

\def\bk{\bar k}\def\bl{\bar l}
\def\bj{\bar j}
\def\bi{\bar i}
\def\b{\bar}
\def\nd{[n]^d}

\title{Sparse Monte Carlo method for nonlocal diffusion problems}
\author{  
Dmitry Kaliuzhnyi-Verbovetskyi and 
Georgi S. Medvedev
\thanks{
Department of Mathematics, Drexel University, 3141 Chestnut Street,
Philadelphia, PA 19104;
{\tt medvedev@drexel.edu}
}
}

\begin{document}
\maketitle

\begin{abstract}
A class of evolution equations with nonlocal diffusion is considered in this work. These are
integro-differential equations arising
as models of propagation phenomena in continuum media with nonlocal interactions
including neural tissue, porous media flow, peridynamics, models with fractional diffusion,
as well as continuum limits of interacting dynamical systems. 
The principal challenge of numerical integration of nonlocal systems stems from the lack of spatial 
regularity of the data and solutions intrinsic to nonlocal models. To overcome this problem 
we propose a semidiscrete numerical scheme based on the combination of sparse Monte Carlo and 
discontinuous Galerkin methods. Our method requires minimal assumptions on the regularity 
of the data. In particular, the kernel of the nonlocal diffusivity is assumed to be a square integrable
function and may be singular or discontinuous. An important feature of our method 
is sparsity. Sparse sampling of points in the Monte Carlo
approximation of the nonlocal term allows to use fewer discretization points without compromising 
the accuracy. For kernels with singularities, more points are selected automatically in the regions
near the singularities.

We prove convergence of the numerical method and estimate the rate of convergence.
There are two principal ingredients in the error of the numerical method related to the use
of Monte Calro and Galerkin approximations respectively. We analyze both errors. 
Two representative examples of discontinuous kernels are presented.
The first example features a kernel with a singularity, while
the kernel in the second example experiences jump discontinuity. We show how the information 
about the singularity in the former case and the geometry of the discontinuity set in the latter
translate into the rate of convergence of the numerical procedure. In addition, we illustrate 
the rate of convergence estimate with a numerical example of an initial value problem, for which 
an explicit analytic solution is available. Numerical results are consistent with analytical estimates.  
\end{abstract}

\section{Introduction}
\setcounter{equation}{0}

We propose a numerical method for the initial value problem (IVP) for a nonlinear heat
equation with nonlocal diffusion
\begin{eqnarray}\lbl{nloc}
\p_t u(t,x) &=& f(u, x, t) +\int W(x,y) D\left(u(t,y)- u(t,x)\right) dy,
\quad x\in Q\subset\R^d,\\
\lbl{nloc-ic}
 u(0,x)&=& g(x).
\end{eqnarray}
For analytical convenience, we take $Q=[0,1]^d$ as a spatial domain. Throughout this paper,
when the domain of integration is not specified, it is assumed to be $Q$.
Further, $g\in L^2(Q)$, $W\in L^2(Q^2)$, $f$ is
a bounded measurable function on $\R\times Q\times \R^+$, which is Lipschitz continuous in 
$u,$ continuous in $t,$ and integrable in $x,$ and $D$ is a Lipschitz continuous function on $\R:$ 
\be\lbl{Lip}
\left| D(u_1)-D(u_2)\right|\le L_D |u_1-u_2|\quad\mbox{and}\quad
\left| f(u_1,x,t)-f(u_2,x,t)\right|\le L_f |u_1-u_2|
\ee
for all $(x,t)\in Q\times \R.$
Throughout this paper we assume
\be\lbl{bound-D}
\sup_{u\in \R} |D(u)|\le 1.
\ee
This  assumption may be dropped if an apriori estimate on $\|u\|_{C(0,T; L^\infty(Q))}$ for $T>0$ is 
available. 
Furthermore, the analysis below applies to models with the interaction function of a more 
general form $D(u_1,u_2)$ provided
\be\lbl{D-Lip-2d}
\left| D(u_1,v_1)-D(u_2,v_2)\right| \le L_D \left( |u_1-u_2|+|v_1-v_2|\right)\quad 
\forall u_1, u_2, v_1, v_2\in \R. 
\ee
However, we keep $D(u_1,u_2):=D(u_2-u_1)$ to emphasize the connection to 
diffusion problems.

Equation \eqref{nloc} is a nonlocal diffusion problem. It arises as a continuum limit of interacting
particle systems \cite{Med18, Gol16}. Equations of this form are used for modeling population 
dynamics \cite{ShenXie15,CarFife05, SchSil16, AMRT10, BouCal10}, neural tissue \cite{CGP14},  
porous media flows \cite{PabQuiRod16, TEJ17}, and various other biological and physicochemical
processes involving anomalous diffusion \cite{Vazquez17, BBPN18}.
The key distinction of the evolution equations with nonlocal diffusion from their classical counterparts is 
the lack of smoothening property. A priori the solution of \eqref{nloc}, \eqref{nloc-ic} is a square integrable
function in $x$ for all $t>0$ \cite{Med18} and it may not possess much more regularity beyond that, 
unless the initial data and kernel $W$ are smooth \cite[Theorem~3.3]{Med14a}. The lack of smoothness
is   a  serious challenge
for constructing numerical schemes for \eqref{nloc}, \eqref{nloc-ic}
and for analyzing their convergence.
All deterministic quadrature formulas require at least piecewise differentiability for a guaranteed convergence
rate. The problem is even more challenging in high dimensional spatial domains. The main idea underlying
our approach is to use the Monte Carlo approximation of the nonlocal term in \eqref{nloc}.
% is to use a combination of the Monte Carlo and discontinuous Galerkin methods to construct a numerical
% scheme for the IVP \eqref{nloc}, \eqref{nloc-ic} that performs well under minimal asumptions on the regualrity
% of $W$ and initial data.
 We take advantage of the essential feature of 
the Monte Carlo method: 
the independence
of the convergence rate on the regularity of the integrand. The second key idea is sparsity, 
whose use is twofold: First, sparse sampling of points in the Monte Carlo method is used to minimize 
computation without compromising the accuracy. For $W$ with jump discontinuity
across Lipschitz hypersurfaces, the use of sparsity is computationally beneficial starting from $d=2$.
If the discontinuity set has nontrivial fractal dimension, the sparse Monte Carlo method performs better
than its dense counterpart already for $d=1$ (cf.~Lemma~\ref{lem.graphon}).
Furthermore, sparsity is the key for extending the Monte
Carlo method for models  with singular kernels  (see \S~\ref{sec.singular}). Not only does it allow to apply
the Monte Carlo method for unbounded functions, it also makes it automatically adaptive: more sample
points are selected near the singularities.
The combination of these ideas together with the
discontinuous Galerkin method yields a numerical
scheme for the IVP \eqref{nloc}, \eqref{nloc-ic} that performs well under minimal assumptions on the regularity
of $W$ and initial data.

This paper is based on our previous work on convergence of interacting particle systems on convergent graph 
sequences \cite{Med14a, Med14b, KVMed17, Med18}. Continuum limit is a powerful tool for studying various
aspects of network dynamics including existence, stability, and bifurcations of spatiotemporal patterns
\cite{WilStr06, Med14c, MedTan15b, MedTan18}. Very often the derivation of the continuum limit is based
on heuristic considerations and its rigorous mathematical justification is a nontrivial problem. Recently,
motivated by the theory of graph limits \cite{LovSze06, LovGraphLim12, BCCZ19}, we proved convergence
to the continuum limit for a broad class of dynamical systems on graphs \cite{Med14a}.
Importantly, our proof applies to models on random graphs \cite{Med14b} including sparse random
graphs \cite{KVMed17, Med18}. These results prepared the ground for the numerical method proposed
in this paper. There is an intimate relation between the problem of the continuum limit for interacting 
particle systems and numerical integration of nonlocal diffusion models. Given a continuum model \eqref{nloc},
one can construct the corresponding particle system, approximating \eqref{nloc}. This idea had been 
already mentioned in \cite{Med14a}, but has never been detailed. Further, recent results for the continuum
limit of coupled systems on sparse graphs indicate a strong potential of sparse discretization for
numerical integration of nonlocal problems. It is the goal of this paper to present these ideas in 
detail.

In the next section, we present a discretization of \eqref{nloc}, which can be viewed as an interacting 
dynamical
system on a \textit{sparse random} graph. The structure of the graph is determined by the kernel $W$, 
which defines 
the asymptotic connectivity of the graph sequence parametrized by the size of the graph. In the theory of graph
limits, such functions are called graphons \cite{LovGraphLim12}, the term we adopt for the reminder of this paper. In Section~\ref{sec.main},
we prove convergence of the semidiscrete (discrete in space and continuous in time) approximation of 
\eqref{nloc} and turn to estimating the rate of convergence in Sections~\ref{sec.examples} and \ref{sec.rate}.
There are two main factors contributing to the error of approximation. The first is due to approximating the 
nonlocal term in \eqref{nloc}
by a random sum (Monte Carlo method), while the second is due to approximating the kernel and the
 initial data by piecewise constant functions (discontinuous Galerkin method). The rate of convergence of the 
sparse Monte Carlo approximation follows from our previous results
\cite[Theorem~4.1]{Med18}. Convergence of piecewise constant approximation in the $L^2$--norm follows 
from classical 
theorems of analysis (cf.~the Lebesgue-Besicovitch Theorem \cite{Evans-fine} or 
$L^2$--Martingale Convergence Theorem \cite{Williams-Prob-Mart}). However, neither of these theorems
elucidates the rate of convergence.  In fact, the example in \S~\ref{sec.graphon} shows that without additional
hypotheses  the algebraic convergence may be arbitrarily slow. 
To this end, we study what determines the rate of convergence of piecewise constant approximations
for a square integrable function. For H{\" o}lder continuous
functions the answer is simple (cf.~Lemma~\ref{lem.Holder}). For discontinuous functions, on the other hand, 
the answer
naturally depends on the type of discontinuity. In Section~\ref{sec.examples},
we consider two examples elucidating this issue. The first example is based on a singular
(unbounded) graphon. It shows how the information about the singularity translates into the 
rate of convergence estimate. Here, we also see how to use sparsity to optimize computation.
The second example adapted from \cite{Med14a}, on the other hand, presents a bounded graphon with jump discontinuity 
(\S~\ref{sec.graphon}). In this case, the convergence rate depends on the geometry of the set of the 
discontinuity (cf.~Lemma~\ref{lem.graphon}). In the light of these examples, in Section~\ref{sec.rate},
we perform convergence analysis under general assumptions on $W$. In Section~\ref{sec.numerics},
we illustrate rate of convergence estimates with a numerical example. Here, we choose a nonlinear
problem, which has an explicit solution. This allows us to verify the rate of convergence of the $L^2$-error
as the discretization step tends to zero. Special attention is paid to the dependence of the convergence rate
on sparsity. Finally, in Section~\ref{sec.average} 
we present a proof of a technical Lemma~\ref{lem.ave}, which extends the corresponding result in 
\cite{Med18} to models in multidimensional domains and affords a wider range of sparsification.

Numerical methods for nonlocal diffusion problems have been subject of intense research recently due to their
increased use in modeling \cite{Du2019a, Du2019b, Du2018, BBPN18, Noch19, Noch16}. Compared to the
existing literature, the contribution of the present work is that our method applies to problems with nonlinear
diffusivity as well as to problems with more general form of the interaction function (cf.~\eqref{D-Lip-2d}). 
The main focus of this paper is how to deal with models with low regularity of the data.  We believe that the 
combination of the Monte Carlo and discontinuous Galerkin methods provides an effective tool for numerical
integration of nonlocal problems under minimal regularity assumptions.
 
\section{The model and its discretization} \lbl{sec.model}
\setcounter{equation}{0}
In this section, we formulate the technical assumptions on the kernel $W$ and 
describe the numerical scheme for solving the IVP \eqref{nloc}-\eqref{nloc-ic}.

We assume that $W\in L^2(Q^2)$ is subject to the following 
assumptions:
\be\lbl{W-1}
\tag{W-1}
\max\left\{ \esssup_{x\in Q} \int |W(x,y)| dy, \;\esssup_{y\in Q} \int |W(x,y)| dx \right\}\le W_1,
\ee
% \be\lbl{W-2}
% \tag{W-2}
% \inf_{x\in Q} \int |W(x,y)| dy=:\nu>0.
% \ee

\begin{thm}\lbl{thm.exists}
Let $W \in L^2(Q^2)$ satisfy \eqref{W-1}. Then for any $g\in L^2(Q)$ and $T>0$ 
there is a unique solution of the IVP \eqref{nloc}, \eqref{nloc-ic} $u\in C^1(0,T; L^2(Q))$.
\end{thm}
\begin{proof}
The proof is as in \cite[Theorem~3.1]{KVMed17} with minor adjustments.
\end{proof}

Next, we note that the kernel in the nonlocal term may be assumed nonnegative.
Indeed, by writing $W=W^+-W^-$ as the difference of its positive and negative parts, one can rewrite
\eqref{nloc} as 
\be\lbl{rewrite-nloc}
\p_t u(t,x) = f(u, x, t) +\int W^+(x,y) D\left(u(t,y)- u(t,x)\right) dy -\int W^-(x,y) D\left(u(t,y)- u(t,x)\right) dy,
\ee
where the nonlocal terms splits into the difference of two terms with nonnegative kernels. Thus, without loss
of generality, in the remainder of this paper we will assume
\be\lbl{nonnegativeW}
W\ge 0.
\ee

We approximate the IVP \eqref{nloc}, \eqref{nloc-ic} by the following system of ordinary 
differential equations
\begin{eqnarray}
\lbl{KM}
\dot u_{n,\bar i} &=& f_{n,\bar i}(u_{n,\bar i},t) + (\alpha_nn^d)^{-1} \sum_{\bar j\in \nd} a_{n,\bar i\bar j} 
D(u_{n,\bar j}-u_{n,\bar i}),
\quad \bar i \in \nd,\\
\lbl{KM-ic}
u_{n,\bar i}(0)&=&g_{n,\bar i}, 
\end{eqnarray}
where
\be\lbl{def-fi}
Q_{n,\bar i} =\left[ {i_1-1\over n}, {i_1\over n}\right)
\times \left[ {i_2-1\over n}, {i_2\over n}\right)
\times \dots \times \left[ {i_d-1\over n}, {i_d\over n}\right), 
\ee
\be\lbl{where}
g_{n, \bar i} = n^d\int_{Q_{n,\bar i}} g(x) dx,\quad 
 f_{n,\bar i}(u,t)= n^d \int_{Q_{n,\bar i}} f(u,x,t) dx, \quad \bar i:=(i_1, i_2,\dots, i_d)\in [n]^d.
\ee
The semidiscrete system \eqref{KM} can be viewed as a system of interacting
particles on a random graph $\Gamma_n$ with the node set $[n]^d$ and adjacency matrix 
$(a_{n,\bar i\bar j})$.  The positive sequence 
\be\lbl{def-alpha}
\alpha_n=n^{-d\gamma},\quad 0\le \gamma< 1,
\ee
is used to control the sparsity of $\Gamma_n$. The adjacency matrix $(a_{n,\bar i\bar j})$ is defined as
follows. 
% From now on we  assume that 
% \be\lbl{nonnegativeW}
% W\ge 0.
% \ee
% This is done to simplify presentation. In Section~\ref{sec.generalize}, we
% show how to construct the discrete scheme without this assumption.
The case $W\in L^\infty(Q^2)$ is slightly different and so we treat it separately.
Thus, there are two cases to consider.
\begin{description}
\item[(I)] Suppose $W\in L^\infty(Q^2)$. Without loss of generality, we 
further assume that $0\le W\le 1$. Then let
\be\lbl{bounded-averageW}
W_{n,\bar i\bar j}=
n^{2d} \int_{Q_{n,\bar i}\times Q_{n,\bar j}}  W(x,y)dxdy,
\ee
and
\be\lbl{boundedP}
\P(a_{n,\bar i\bar j}=1)=\alpha_n W_{n,\bar i\bar j},\quad \bi,\bj \in \nd.
\ee
If $\alpha_n\equiv 1 (\gamma=0)$ (cf.~\eqref{def-alpha}), $\Gamma_n$ is a dense
W-random graph \cite{LovSze06}, otherwise $\Gamma_n$ is sparse with the 
mean degree $O(n^{d(1-\gamma)}), \; 0\le \gamma<1.$

\item[(II)] Alternatively, if $W$ is in $L^2(Q^2)$ but not in $L^\infty(Q^2)$ then let
\be\lbl{cutoff}
\tilde W_n(x,y):= \alpha_n^{-1} \wedge W(x,y)\quad\mbox{and}\quad
W_{n,\bar i\bar j}=
n^{2d} \int_{Q_{n,\bar i}\times Q_{n,\bar j}}  \tilde W(x,y)_n dxdy,
\ee
where $\alpha_n$ defined in \eqref{def-alpha} with $\gamma\in (0,1)$.
Then
\be\lbl{Pedge}
\P(a_{n,\bar i\bar j}=1) = \alpha_n W_{n,\bar i,\bar j},\quad {\bar i}, {\bar j}\in \nd.
\ee
\end{description}

\section{Convergence of the numerical method} \lbl{sec.main}
\setcounter{equation}{0}

In this section, we study convergence of the discrete scheme \eqref{KM}, \eqref{KM-ic}.
We first deal with the more general case of unbounded graphon $W$ and then specialize the
result for $W\in L^\infty(Q^2)$.

The following additional mild assumption on $W$ is used to get a wider range of sparsity.
Let nonnegative $W\in L^4(Q^2)$ satisfy
\be\lbl{W-1s}
\tag{W-1s}
\max\left\{  \esssup_{x\in Q} \int W^k (x,y) dy,\; \esssup_{y\in Q} \int W^k (x,y) dx
\right\} \le \bar W_k,\quad k\in[4].
\ee

\begin{thm}\lbl{thm.main}
 Suppose nonnegative $W\in L^4(Q^2)$ is subject to \eqref{W-1s}, $D,$ $f,$ and $g$ are as in
\eqref{nloc}, \eqref{nloc-ic}. Further, $\alpha_n=n^{-d\gamma}$ for some $\gamma\in (0,1)$.
Then for arbitrary
$0<\delta<1-\gamma,$ we have
\be\lbl{main-est}
\begin{split}
\sup_{t\in [0,T]}  \|u(t,\cdot)-u_n(t, \cdot)\|_{L^2(Q)}  &\le C
\left( \|g-g_n\|_{L^2(Q)}+ \sup_{u\in\R, \,t\in [0,T]} \|f(u,\cdot, t)-f_n(u,\cdot, t) \|_{L^2(Q)}\right.\\
&\left.+ \| \tilde W_n-W\|_{L^2(Q^2)} +
\| \tilde W_n- P_n \tilde W_n\|_{L^2(Q^2)} +
n^{-d(1-\gamma-\delta)/2}\right),
\end{split}
\ee
where $C$ is a positive constant independent of $n$, and  
$P_n\tilde W_n=:W_n$ stands for the $L^2$-projection of $\tilde W_n$
onto the finite--dimensional subspace $X_n=\operatorname{span} 
\{ \1_{Q_{n,\bi}\times Q_{n,\bj}},\, (\bi,\bj)\in [n]^{2d}\}$:
$$
W_n(x,y)=\sum_{(\bar i, \bar j)\in {\nd}^2} W_{n,\bar i\bar j} 
\1_{Q_{n,\bar i}\times Q_{n,\bar j}} (x,y).
$$ 
and 
$$
f_n(u,x,t)=\sum_{\bar i\in \nd}  f_{n,\bar i}(u, t)\1_{Q_{n,\bar i}}(x).
$$
Estimate \eqref{main-est} holds almost surely (a.s.) with respect to the random graph model.
\end{thm}
\begin{rem} The theorem still holds without  \eqref{W-1s}, i.e., for 
 square integrable $W$ subject to \eqref{W-1}. In this case, the last term on the right-hand side of
 \eqref{main-est} is replaced by $n^{-d(1/2-\gamma-\delta)},$ $\gamma\in (0,1/2),$ and $\delta<1/2-\gamma.$
\end{rem}

The first two terms on the right--hand side of \eqref{main-est} correspond to the error of approximation
of the initial data $g\in L^2(Q)$ and $f(u,x,t)$ by the step functions in $x$. Further, 
$\| \tilde W_n-W\|^2_{L^2(Q^2)}$ and $\| \tilde W_n- P_n \tilde W_n\|^2_{L^2(Q^2)}$ bound the error
of approximation of $W$ by a bounded step function $W_n$.
Here, the first term $\| \tilde W_n-W\|^2_{L^2(Q^2)}$ is the error of truncating $W$ and the second term
$\| \tilde W_n- P_n \tilde W_n\|^2_{L^2(Q^2)}$ is the error of approximation of the truncated function
$\tilde W_n$ by projecting it onto a finite--dimensional subspace.
Finally, the last term on the right--hand side of \eqref{main-est} is the error of the approximation 
of the nonlocal term by the random sum in \eqref{KM}.

For bounded graphons $W,$ Theorem~\ref{thm.main} implies the following result.
\begin{cor}\lbl{cor.main}
 Let $W\in L^\infty(Q^2)$. Then under the assumptions of Theorem~\ref{thm.main} we have
\be\lbl{main-est-bounded}
\begin{split}
\sup_{t\in [0,T]}  \|u(t,\cdot)-u_n(t, \cdot)\|_{L^2(Q)}  &\le C
\left( \|g-g_n\|_{L^2(Q)}+ \sup_{u\in\R, \,t\in [0,T]} \|f(u,\cdot, t)-f_n(u,\cdot, t) \|_{L^2(Q)}\right.\\
&\left.+ \| W_n- P_n W_n\|_{L^2(Q^2)} +
n^{-d(1-\gamma-\delta)/2}\right)\qquad\mbox{a.s.}.
\end{split}
\ee
where $C$ is a positive constant independent of $n.$
\end{cor}
\begin{rem}\lbl{rem.important}
From \eqref{main-est-bounded} one can see how to use sparsity to optimize computation.
Already for  $d=1,$ if the largest of the two errors of
approximation of $g$ and $W$ by  step functions is $O(n^{-\kappa})$ with $1/2<\kappa<1$ 
(cf. \S\ref{sec.graphon}) and the nonlinearity $f(\cdot)$ does not  depend on $x,$ then 
taking $\gamma=1-2\kappa$ one can use sparse discretization without compromising the 
accuracy. This has obvious computational advantages over dense random and, moreover, 
deterministic spatial discretization schemes, e.g., Galerkin method. 
Sparse random discretization is even more efficient when $d>1$.
\end{rem}

The proof of Theorem~\ref{thm.main} modulo a few minor details proceeds as the proof
of convergence to the continuum limit in \cite{Med14b, Med18}.
First, the solution of the IVP \eqref{KM}, \eqref{KM-ic} is compared to that of the IVP for 
the averaged equation:
\begin{eqnarray}
\lbl{aKM}
\dot v_{n,\bar i} &=& f_{n,\bar i}(v_{n,\bar i},t) + n^{-d} \sum_{\bar j\in \nd} 
W_{n,\bar i\bar j} D(u_{n,\bar j}-u_{n,\bar i}),
\quad \bar i\in [n]^d,\\
\lbl{aKM-ic}
v_{n,\bar i}(0)&=&g_{n,\bar i}. 
\end{eqnarray}
Then the solution of the averaged problem is compared to the solution of the 
IVP \eqref{nloc}, \eqref{nloc-ic}. It is convenient to view the
solution of the averaged problem as a function on $\R^+\times Q$:
\be\lbl{step-v}
v_n(t,x) =\sum_{\bar i\in \nd} v_{n,\bar i}(t) \1_{Q_{n,\bar i}}(x).
\ee
Likewise, we interpret the solution of the discrete problem \eqref{KM}, \eqref{KM-ic} 
as a function on $\R^+\times Q:$
\be\lbl{step-u}
u_n(t,x) =\sum_{\bar i\in \nd} u_{n,\bar i}(t) \1_{Q_{n,\bar i}}(x).
\ee

We recast the IVP \eqref{aKM}, \eqref{aKM-ic} as follows
\begin{eqnarray}
\lbl{re-aKM}
\p_t v_n(t,x)&=& f_n(v_n(t,x),x,t)+ \int W_n(x,y) D\left(v_n(t,y)-v_n(t,x)\right) dy,\\
\lbl{re-aKM-ic}
v_n(0,x)&=&g_n(x).
\end{eqnarray}

The first step of the proof of convergence of the numerical scheme \eqref{KM}, \eqref{KM-ic}
is accomplished in the following lemma.
\begin{lem}\lbl{lem.ave}
Let nonnegative $W\in L^4(Q^2)$ subject to \eqref{W-1s},  and 
$\alpha_n=n^{-d\gamma},\; \gamma\in (0, 1)$ (cf.~\eqref{cutoff}).
%and 
% \be\lbl{def-L}
%L=L_f+4 L_D+{1\over 2}.
%\ee
Then for any $T>0$ for solutions of  \eqref{KM} and \eqref{aKM} subject to the 
same initial conditions,  
% and any $T\le C\ln n$, $0\le C< (1-2\gamma)L^{-1}$,
we have
\be\lbl{ave.statement}
\sup_{t\in [0,T]} \| u_n(t,\cdot) -v_n(t,\cdot)\|_{L^2(Q)} \le 
C n^{-d(1-\gamma-\delta)/2}\quad \mbox{a.s.},
\ee
where $0<\delta<1-\gamma,$ and positive constant $C$
independent of $n$. 
\end{lem}
The proof of the lemma is technical and is relegated to Section~\ref{sec.average}.
The result still holds without for square integrable functions the additional assumption  
\eqref{W-1s} albeit for a narrower range of $\gamma\in (0,0.5)$ (cf.~\cite[Theorem~4.1]{Med18}).

\begin{proof}[Proof of Theorem~\ref{thm.main}]
Denote the difference between the solutions of the original IVP \eqref{nloc}, \eqref{nloc-ic}
and the averaged IVP \eqref{re-aKM}, \eqref{re-aKM-ic}
\be\lbl{difference}
w_n(t,x)=u(t,x)-v_n(t,x).
\ee
By subtracting \eqref{aKM} from \eqref{nloc}, multiplying the resultant equation by $w_n$, and 
integrating over $Q$, we obtain
\be\lbl{energy}
\begin{split}
\int \p_t w_n(t,x) w_n(t,x)dx &= \int \left(f(u(t,x), x,t) -f(v_n(t,x),x,t)\right) w_n(t,x)dx \\
&+ \int \left(f(v_n(t,x), x,t) -f_n(v_n(t,x),x,t)\right) w_n(t,x)dx\\
&+ \int\int W(x,y) \left[ D(u(t,y)-u(t,x)) - D(v_n(t,y)-v_n(t,x))\right]  w_n(t,x)dydx \\
&+ \int\int \left( W(x,y)-W_n(x,y)\right) D(v_n(t,y)-v_n(t,x))   w_n(t,x)dydx.
\end{split}
\ee 

Using Lipschitz continuity of $f(u,x,t)$ in $u$ and an elementary case of the Young's inequality,
we obtain
\be\lbl{part-1a}
\left|\int \left(f(u(t,x), x,t) -f(v_n(t,x),x,t)\right) w_n(t,x)dx\right|\le L_f
\int w_n(t,x)^2 dx,
\ee
\be\lbl{part-1b}
\begin{split}
\left|\int \left(f(v_n(t,x), x,t) -f_n(v_n(t,x),x,t)\right) w_n(t,x)dx\right|& \le
{1\over 2} \int \left(f(v_n(t,x), x,t) -f_n(v_n(t,x),x,t)\right)^2dx \\
&+{1\over 2 }\| w_n(t,\cdot)\|^2,
\end{split}
\ee
where $\| \cdot \|$ stands for the $L^2(Q)$-norm.
Recall that $D$ is bounded by $1$ (cf.~\eqref{bound-D}).
Using this bound and the Young's inequality, we obtain
\be\lbl{part-2}
\begin{split}
\left| \int\int \left( W(x,y)-W_n(x,y)\right) D(v_n(t,y)-v_n(t,x))   w_n(t,x)dydx\right|
&\le {1\over 2} \left\|W-W_n\right\|_{L^2(Q^2)}\\
& +{1\over 2} \left\|w_n\right\|^2.
\end{split}
\ee

Finally,  using Lipschitz continuity of $D$ and Young's inequality, we estimate
\be\lbl{part-3}
\begin{split}
 &\left|\int\int W(x,y) \left[ D(u(t,y)-u(t,x)) - D(v_n(t,y)-v_n(t,x))\right]  w_n(t,x)dydx \right| \\
  &\le L_D \int\int W(x,y) \left( |w_n(t,y)| +|w_n(t,x)|\right) |w_n(t,x)|dydx \\
&\le L_D \int\int W(x,y) \left( {1\over 2} |w_n(t,y)|^2 +{3\over 2} |w_n(t,x)|^2\right) dydx \\
& \le  {3L_D\over 2} \int\int W(x,y) |w_n(t,x)|^2 dydx  + 
{L_D\over 2} \int\int W(x,y) |w_n(t,y)|^2 dydx\\
&\le 2W_1 L_D \left\|w_n\right\|^2,
\end{split}
\ee
where we used  Fubini theorem and \eqref{W-1} in the last line.

By combining \eqref{energy}-\eqref{part-3}, we arrive at
\be\lbl{pre-Gron}
{d\over dt} \|w_n(t,\cdot)\|^2 \le L \|w_n(t,\cdot)\|^2 +
\sup_{u\in\R, \, t\in [0,T]} \|f(u,\cdot, t)-f_n(u,\cdot, t) \|^2
+ \|W_n-W\|^2,
\ee
where $L=1+2L_f +L_D(3W_1+W_2).$

By Gronwall's inequality, we have
\begin{equation*}
\begin{split}
\sup_{t\in [0,T]}  \|w_n(t, \cdot)\|  &
\le e^{LT/2} \sqrt{ \|w_n(0,\cdot)\|^2 +
\sup_{u\in\R, \, t\in [0,T]} \|f(u,\cdot \, t)-f_n(u,\cdot, t) \|^2
+ \| W_n-W\|^2_{L^2(Q^2)}
}\\
& \le e^{LT/2} \left(\|g-g_n\|_{L^2(Q)}+
\sup_{u\in\R, \, t\in [0,T]} \|f(u,\cdot,t)-f_n(u,\cdot,t) \|
+ \| W_n-W\|_{L^2(Q^2)}
\right).
\end{split}
\end{equation*}
\end{proof}

\section{Two examples}\lbl{sec.examples}
The error of approximation of the nonlocal term by a random sum, the last term on the 
right--hand side of \eqref{main-est}, is known explicitly. Next in importance is the error 
of approximation of the square integrable graphon $W$ by the step function $W_n.$
This error depends on the regularity of the graphon. In this section, we consider two representative 
examples of $W$: a singular graphon (\S~\ref{sec.singular}) and a bounded graphon
with jump discontinuities (\S~\ref{sec.graphon}). Motivated by these examples in the next section,
we will analyze the rate of convergence estimates under general assumptions on graphon $W$.

We will begin with the following estimate for H\"{o}lder continuous functions. 
To this end,
$\phi\in L^p(Q), \; p\ge 1,$ and 
$$
\phi_n(x)=\sum_{\bi\in [n]^d} \phi_{n,\bi} \1_{Q_{n,\bi}}(x), \quad \phi_{n,\bi}=n^d\int_{Q_{n,\bi}} \phi(x) dx.
$$
\begin{lem}\lbl{lem.Holder}
Suppose $\phi\in L^p(Q), \; p\ge 1,$ is a H\"{o}lder continuous function
\be\lbl{f-Holder}
|\phi(x)-\phi(y)|\le C |x-y|^\beta, \quad x,y \in Q, \; \beta\in (0,1].
\ee
Then
\be\lbl{Holder}
\|\phi-\phi_n\|_{L^p(Q)}\le C h^{\beta}.
\ee
Here and below, $h:=n^{-1}$.
\end{lem}
\begin{proof}
Using Jensen's inequality and \eqref{Holder}, we have
\be\lbl{delta-p}
\begin{split}
 \|\phi-\phi_n\|^p_{L^p(Q)} & = \int_Q \left|  \sum_{\bar i\in \nd}  
\left(  \phi(x) -n^{d}\int_{Q_{n,\bar i}}  \phi(y) dy \right) \1_{Q_{n,\bar i}}(x) \right|^p 
dx \\
&=
  \sum_{\bar i\in \nd}   \int_{Q_{n,\bar i}} 
\left| n^{d}\int_{Q_{n,\bar i}}\left( \phi(x) - \phi(y)\right) dy\right|^p dx\\
&\le n^{d} \sum_{\bar i\in \nd} \int_{Q_{n,\bar i}}\int_{Q_{n,\bar i}}
\left|\phi(x) - \phi(y)\right|^p dxdy\\
& \le C^p h^{p\beta}.
\end{split}
\ee
\end{proof}
\begin{rem}\lbl{rem.Q2}
Below, we will freely apply Lemma~\ref{lem.Holder} to functions on $Q$ and on $Q^2$. The latter 
are clearly covered by the lemma by taking $Q:=Q^2$.
\end{rem}

\subsection{A singular graphon}\lbl{sec.singular}

Consider the problem of approximation by step functions of the 
singular kernel graphon
\be\lbl{Wgamma}
W(x,y) = {1\over |x-y|^\lambda}, \quad x,y \in Q=[0,1]^d,
\ee
where $0<\lambda<d/2$.

\begin{lem}\lbl{lem.example} For $\gamma\in (0,1/2)$ and $0<\lambda<d/2$ we have
\be\lbl{example.error}
\|W-W_n\|_{L^2 (Q^2)} \le \max\left\{ O\left( h^{d\gamma ({d\over 2\lambda}-1) } \right),
O\left(h^{1-d\gamma\left(1+{1\over\lambda}\right) }\right) \right\}.
\ee
\end{lem}
\begin{proof}
\begin{enumerate}
\item Below, we will use the following change of variables
$(x,y)= T(u,v)$ for $(x,y)$ and $(u,v)$ from $\R^{2d},$
defined by
\be\lbl{Tuv}
u_i=x_i-y_i \quad\mbox{and}\quad v_i=x_i+y_i,\; i\in [d].
\ee
\item Let $\alpha_n=n^{-d\gamma}$ and recall that $\tilde W_n=h^{-d\gamma}\wedge W$.
Denote $\tilde Q=\left\{ (x,y)\in Q^2:\; |x-y|^{-\lambda}\ge h^{-d\gamma}\right\}$. 
Further,
\be\lbl{W-tildeW}
\begin{split}
\|W-\tilde W_n \|^2_{L^2(Q^2)} &= \int_{\tilde Q^2} 
\left( {1\over |x-y|^{\lambda}} -n^{d\gamma}\right)^2 dxdy\\
&\le C_1 \int_{ \{ |u| \le  h^{{d\gamma\over \lambda}} \} } 
\left({1\over |u|^{\lambda}}-n^{d\gamma}\right)^2 du\\
&\le C_2 \int_0^{ h^{d\gamma\over \lambda} } 
\left( {1\over r^\lambda}-n^{d\gamma}\right)^2 r^{d-1} dr\\
&=O\left((d-2\lambda)^{-1} h^{2d\gamma \left({d\over 2\lambda}- 1\right)}\right),
\end{split}
\ee
where we used \eqref{Tuv} followed by the change to polar coordinates.

Thus,
\be\lbl{cutoff}
\|W-\tilde W_n \|_{L^2(Q^2)}=O\left( h^{d\gamma \left({d\over 2\lambda}- 1\right)} \right), 
0<\lambda<d/2,\; \gamma>0.
\ee

\item
Next we turn to estimating $\left\| \tilde W_n- W_n \right\|_{L^2(Q^2)}$.  Since the truncated function 
$\tilde W_n$ is Lipschitz continuous on $Q^2$, by Lemma~\ref{lem.Holder},
$$
\left\| \tilde W_n- W_n \right\|_{L^2(Q^2)}\le L(\tilde W_n)h.
$$
It remains to estimate the Lipschitz constant $L(\tilde W_n)\le \esssup_{Q^2}|\nabla \tilde W_n|.$
On $Q^2-\tilde Q$, 
\be\lbl{gradient}
|\nabla \tilde W_n|=\lambda |x-y|^{-1-\lambda}\left|\nabla_{x,y} |x-y|\right|=\sqrt{2} 
\lambda |x-y|^{-1-\lambda}.
\ee
The gradient approaches its greatest value as $|x-y|\searrow h^{d\gamma\over\lambda}$.
Thus,
$$
 \esssup_{Q^2}|\nabla \tilde W_n| =\sqrt{2}\lambda  h^{-d\gamma\left({1\over \lambda}+1\right)}
$$
and
\be\lbl{second-estimate}
\left\| \tilde W_n- W_n \right\|_{L^2(Q^2)}=O\left(h^{1-d\gamma\left(1+{1\over\lambda}\right)}\right).
\ee
\item The statement of the lemma follows \eqref{cutoff} and \eqref{second-estimate} and the
triangle inequality.
\end{enumerate}
\end{proof}

Next we choose $\gamma$ to optimize the rate 
of convergence in \eqref{example.error}. By setting the two exponents of $h$ on the right--hand side 
of  \eqref{example.error} equal, we see that the rate 
is optimal for 
$$
\gamma={2\lambda\over d(d+2)}.
$$
With this choice of $\gamma,$
$$
 \|W-W_n\|_{L^2(Q^2)}=O(h^{{2\lambda\over d+2}\left({d\over 2\lambda} - 1\right)}).
$$
To optimize the rate of convergence of the numerical scheme \eqref{KM}, \eqref{KM-ic},
one has to choose $\gamma\in (0,1)$ to maximize the smallest of the following three exponents
$$
d\gamma\left({d\over 2\lambda}-1\right), \quad
1-d\gamma\left(1+{1\over \lambda}\right),\quad 
{d\over 2}\left(1-\gamma\right),
$$
where the last exponent comes from the error of the Monte Carlo approximation (cf.~\eqref{main-est}).

\subsection{$\{0,1\}$-valued functions}\lbl{sec.graphon}
The following example is adapted from \cite{Med14a}. It shows how  jump discontinuities affect 
the rate of convergence of approximation by piecewise constant functions.
The accuracy of approximation depends on the geometry of the hypersurface of 
discontinuity, more precisely, on its fractal dimension.

Let $Q^+$ be a closed subset of $Q$ and consider 
\be\lbl{graphon}
f(x)=\left\{ \begin{array}{ll}
1,&  x\in Q^+\\
0,& \mbox{otherwise}.
\end{array}
\right.
\ee
Denote by $\p Q^+$ the boundary of $Q$ and recall the upper box-counting dimension of $\p Q^+$
\be\lbl{box}
\beta:=\overline{\lim}_{h\to 0} {\log N_h (\p Q^+)\over -\log h},
\ee
where $N_h (\p Q^+)$ stands for the number of $Q_{n,\bi},\; \bi, \bar j \in\nd,$ having
nonempty intersection with $\p Q^+$ (cf.~\cite{Falc-FracGeometry}). 

\begin{lem}\lbl{lem.graphon}
\be\lbl{d-gamma}
\|\phi-\phi_n\|_{L^p(Q)}\le C h^{{d-\beta\over p}},
\ee
for some positive $C$ independent on $n$.
\end{lem}
\begin{proof}
As in \eqref{delta-p}, we have
\begin{equation}\lbl{Jensen2}
\|\phi-\phi_n\|_{L^p(Q^2)}^p 
\le h^{-d} \sum_{\bi\in\nd} \int_{Q_{n,\bi}} \int_{Q_{n,\bi}} \left| f(x)-f(z) \right|^pdzdx.
\end{equation}
Note that the only nonzero terms in the sum on the right--hand side of \eqref{Jensen2}
are the integrals over $Q_{n,\bi}\times Q_{n,\bar j}$'s having nonempty intersection with $\p Q^+.$
Thus,
\be\lbl{direct}
\|\phi-\phi_n\|_{L^p(Q)}^p = h^{d} N_h(\p Q^+) \le C h^{d-\beta},
\ee
where we used \eqref{box}.
\end{proof}

\begin{rem} 
Note that as $\beta\to d-0$ the rate of convergence in \eqref{d-gamma}
can be made arbitrarily low.
\end{rem}

\section{The rate of convergence of the numerical method}\lbl{sec.rate}
\setcounter{equation}{0}
\subsection{Approximation by step functions}
In this section, we address the rate of convergence of the Galerkin component of the 
numerical scheme \eqref{KM}, \eqref{KM-ic}. Specifically, we study the error of the 
approximation 
of the graphon $W\in L^2(Q^2)$ and the initial data $g\in L^2(Q)$
by step functions. 

We will need an $L^p$-modulus of continuity of function on a unit $d$-cube $Q=[0,1]^d$.  
In fact, we only need the $L^2$-modulus of continuity, but present the analysis in 
the more general $L^p$-setting, since this does not require any extra effort.
For functions on the 
real line, the definition of the $L^p$-modulus of continuity can be found in \cite{Akh-Approx,
DeVore-book}. Here, we present a suitable adaptation of this definition for the problem at hand.
\begin{df}\lbl{df.Lip}
For $\phi\in L^p(Q)$
we define the $L^p$-modulus of continuity 
\be\lbl{Lpmodulus}
\omega_{p}(\phi,\delta)=\sup_{|\xi|_\infty \le \delta} \|\phi(\cdot+\xi)-\phi(\cdot)\|_{L^p(Q_\xi)},\; 
\delta>0,
\ee
where $|\xi|_\infty:=\max_{i\in [d]} |\xi_i|$ and  $Q_\xi=\{x\in Q:\; x+\xi\in Q\}$. 
\end{df}

For $\alpha\in (0,1],$ we define a generalized Lipschitz space\footnote{
Below, we will freely apply the definitions and various estimates established for
functions on $Q$ to functions on $Q^2,$ for which they are trivially valid by setting $d:=2d$.
In particular, the definitions of 
the modulus of continuity and the corresponding Lipschitz spaces obviously translate to functions
on $Q^2$.}
\be\lbl{Lip-space}
\operatorname{Lip}\left(\alpha, L^p(Q)\right)=\left\{ \phi\in L^p(Q):\; \exists 
C>0 :\; \omega_p(\phi,\delta)\le C\delta^\alpha\right\}.
\ee
Clearly $\operatorname{Lip}\left(\alpha, L^p(Q)\right)$ contains 
$\alpha$-H\"{o}lder continuous functions. However, Lipschitz spaces are much larger than H\"{o}lder
spaces. For instance, $\operatorname{Lip}\left(1/p, L^p(Q)\right)$ contains discontinuous functions.

Below, we express the error of approximation of $\phi\in L^p(Q)$ by a step 
function through $\omega_{p}(\phi, h)$. 
The analysis works out a little cleaner for dyadic discretization of $Q$, which will be assumed for the 
remainder of this section. Thus, we approximate 
$\phi\in L^p(Q)$ by  a piecewise constant function
\be\lbl{f-2m}
\phi_{2^m}(x)=\sum_{\bar i\in [2^m]^d} \phi_{Q_{2^m,\bar i}} \1_{Q_{2^m,\bar i}} (x),
\ee
where $\phi_{Q_{2^m,\bar i}}$ stands for the mean value of $\phi$ on $Q_{2^m,\bar i}$
$$
\phi_{Q_{2^m,\bar i}}(x) =2^{md} \int_{Q_{2^m,\bar i}} \phi (x) dx, \; i\in [2^m].
$$

\begin{lem}\lbl{lem.dyadic}
For $\phi\in\operatorname{Lip}\left(\alpha, L^p(Q)\right)$, we have
\be\lbl{dyadic}
\left\|\phi-\phi_{2^m}\right\|_{L^p(Q)} \le C2^{-\alpha m},
\ee
where $C$ is independent of $m$.
\end{lem}
\begin{proof}
Fix $m\in\N$ and denote $h:=2^{-m}$. To simplify notation, throughout the proof
we drop $2^m$ in the subscript of  $x_{2^m,i}$ and $Q_{2^m, \bar i}$.

We write
\be\lbl{f-2m+1}
\phi_{2^{m+1}}(x)=\sum_{\bar i\in [2^m]^d} \sum_{\bar j\in \{0,1\}^d} 
\left( {2\over h} \right)^d \int_{Q^\prime_{\bar i}} 
\phi\left( s+ \bar j {h\over 2}\right) ds \1_{Q_{\bar i}^{\bar j}}(x),
\ee
where 
$$
s+ \bar j {h\over 2} =\left( s_1+j_1 {h\over 2}, s_2+j_2 {h\over 2},\dots, s_d+j_d {h\over 2}
\right),
$$
$$
{Q_{\bar i}^{\bar j}}=\left[ x_{i_1-1}+j_1 {h\over 2} , x_{i_1-1}+(j_1+1) {h\over 2} \right)
% \times \left[ x_{i_2-1}+j_2 {h\over 2} , x_{i_2-1}+(j_2+1) {h\over 2} \right)
\times\dots\times
\left[ x_{i_d-1}+j_d {h\over 2} , x_{i_d-1}+(j_d+1) {h\over 2} \right),
$$
and 
$$
{Q_{\bar i}^\prime}=\left[ x_{i_1-1}, x_{i_1-1}+ {h\over 2} \right)
% \times \left[ x_{i_2-1}+j_2 {h\over 2} , x_{i_2-1}+(j_2+1) {h\over 2} \right)
\times\dots\times
\left[ x_{i_d-1},  x_{i_d-1}+{h\over 2} \right).
$$

Rewrite \eqref{f-2m}
\be\lbl{newf2m}
\phi_{2^{m}}(x) =\sum_{\bar i\in [2^m]^d} \sum_{\bar j\in \{0,1\}^d} \left( {1\over h} \right)^d 
\left\{ \sum_{\bar k\in \{0,1\}^d}
\int_{Q^\prime_{\bar i}} 
\phi\left( s+ \bar k {h\over 2}\right) ds \right\}\1_{Q_{\bar i}^{\bar j}}(x).
\ee

By subtracting \eqref{f-2m+1} from \eqref{newf2m}
we have 
\begin{equation*}
\phi_{2^m}-\phi_{2^{m+1}} = 
\sum_{\bar i\in [2^m]^d} \sum_{\bar j \in \{0,1\}^d} \,{1\over h^d}\left\{
\sum_{\stackrel{\bar k \in \{0,1\}^d} {\bar k\neq\bar j}}
\int_{Q^\prime_{\bar i}} 
\left[  \phi \left( s+ \bar k {h\over 2} \right)- \phi \left( s+ \bar j {h\over 2} \right) \right] ds\right\}\,
\1_{Q^{\bar j}_{\bar i}}.
\end{equation*}
Further,
\begin{equation} \lbl{delta-f-p}
\left| \phi_{2^m}-\phi_{2^{m+1}} \right|^p \le  
\sum_{\bar i\in [2^m]^d} \sum_{\bar j \in \{0,1\}^d} \,
\sum_{\stackrel{\bar k \in \{0,1\}^d} {\bar k\neq\bar j}}
\left| \int_{Q^\prime_{\bar i}} 
\left[ h^{-d} \phi\left( s+ \bar k {h\over 2} \right)- \phi\left( s+ \bar j {h\over 2} \right) \right] ds\right|^p 
\; \1_{Q^{\bar j}_{\bar i}}.
\end{equation}

Integrating both sides of \eqref{delta-f-p} over $Q$ and using
Jensen's inequality, we continue
\begin{equation} \lbl{delta-f-Lp}
\begin{split}
\|\phi_{2^m}-\phi_{2^{m+1}}\|^p_{L^p(Q)} & \le 
\sum_{\bar j \in \{0,1\}^d} \,
\sum_{\stackrel{\bar k \in \{0,1\}^d} {\bar k\neq\bar j}}\,
\sum_{\bar i\in [2^m]^d} \int_{Q^\prime_{\bar i}} 
\left|  \phi \left( s+ k {h\over 2} \right)- \phi\left( s+ \bar j {h\over 2} \right) \right|^p ds\\
& \le 2^d(2^d-1) \omega^p_p(\phi, h).
\end{split}
\end{equation}

Thus,
$$
\|\phi_{2^m}- \phi_{2^{m+1}}\|_{L^p(Q)}\le \left[2^d(2^d-1)\right]^{1/p} \omega_p(\phi, 2^{-(m+1)})
=:C_{d,p} \omega_p(\phi, 2^{-(m+1)}).
$$
Since $\phi\in \operatorname{Lip}\left(\alpha, L^p(Q)\right)$, we have
\be\lbl{subtract-m}
\|\phi_{2^m}- \phi_{2^{m+1}}\|_{L^p(Q)}\le C 2^{-\alpha m},
\ee
where $C$ depends on $\phi, d,$ and $p$ but not $m$.

Let $m\in\N$ be arbitrary but fixed. For any integer $M>m$ we have
\begin{equation}\lbl{bound-dyadic}
\begin{split}
\|\phi_{2^M}-\phi_{2^m}\|_{L^p(Q)} &
=\left\| \sum^{M-1}_{k=m}\left( \phi_{2^{k+1}}- \phi_{2^k}\right)\right\|_{L^p(Q)}
\le \sum_{k=m}^\infty \left\| \phi_{2^{k+1}}- \phi_{2^k}\right\|_{L^p(Q)}\\
&=\left\| \sum^\infty_{k=m}\left( \phi_{2^{k+1}}- \phi_{2^k}\right)\right\|_{L^p(Q)}
\le \sum_{k=m}^\infty \left\| \phi_{2^{k+1}}- \phi_{2^k}\right\|_{L^p(Q)}\\
&\le  \sum^\infty_{k=m} \omega_p(\phi, 2^{-(k+1)})
\le 2^{-p+1} \sum^\infty_{k=m} C 2^{\alpha(k+1)} 
\le C 2^{-\alpha m}.
\end{split}
\end{equation}
By passing $M$ to infinity in \eqref{bound-dyadic}, we get \eqref{dyadic}.
\end{proof}

\subsection{The rate of convergence}
We now can combine Theorem~\ref{thm.main} and Lemma~\ref{dyadic} to estimate 
the convergence rate for \eqref{KM}, \eqref{KM-ic}. 
For the model with a bounded graphon $W$
(cf.~\textbf{(I)}, Section~\ref{sec.model}) we have the following theorem. 

\begin{thm}\lbl{thm.rate}
Suppose that in addition to the assumptions of Theorem~\ref{thm.main}, for some $\alpha_i\in (0,1],\; i\in [3],$
$g\in \operatorname{Lip}\left(\alpha_1, L^2(Q)\right),$ 
$W\in\operatorname{Lip}\left(\alpha_2, L^2(Q^2)\right)\cap L^\infty(Q^2),$ and $f(u,\cdot,t)\in \operatorname{Lip}\left(\alpha_3, L^2(Q)\right)$
uniformly for $(u,t)\in \R\times [0,T],$ i.e.,
$$
\omega_2 (f(u,\cdot,t),\delta)\le C\delta^{\alpha_3}
$$ 
where $C>0$ is independent of $(u,t).$

Then
\be\lbl{rate}
\sup_{t\in [0,T]} \| u(t,\cdot)-u_n(t,\cdot)\|\le Cn^{-\alpha}, \quad \alpha=\min\{\alpha_1, \alpha_2,\alpha_3, {1\over 2}-\gamma\},
\ee
where $C$ is independent of $n$.
\end{thm}

If $W\in L^2(Q^2)$ has singularities then
the convergence rate may also depend on the accuracy of approximation of $W$ by the truncated function 
$\tilde W_n$. We do not estimate the truncation error for a general $W\in L^2(Q^2)$. For an example
of how this error can be estimated for a given graphon in practice, we refer to the 
example in \S~\ref{sec.singular}. 

\begin{figure}
\begin{center}
\textbf{a}\includegraphics[height=1.5in, width=2.0in]{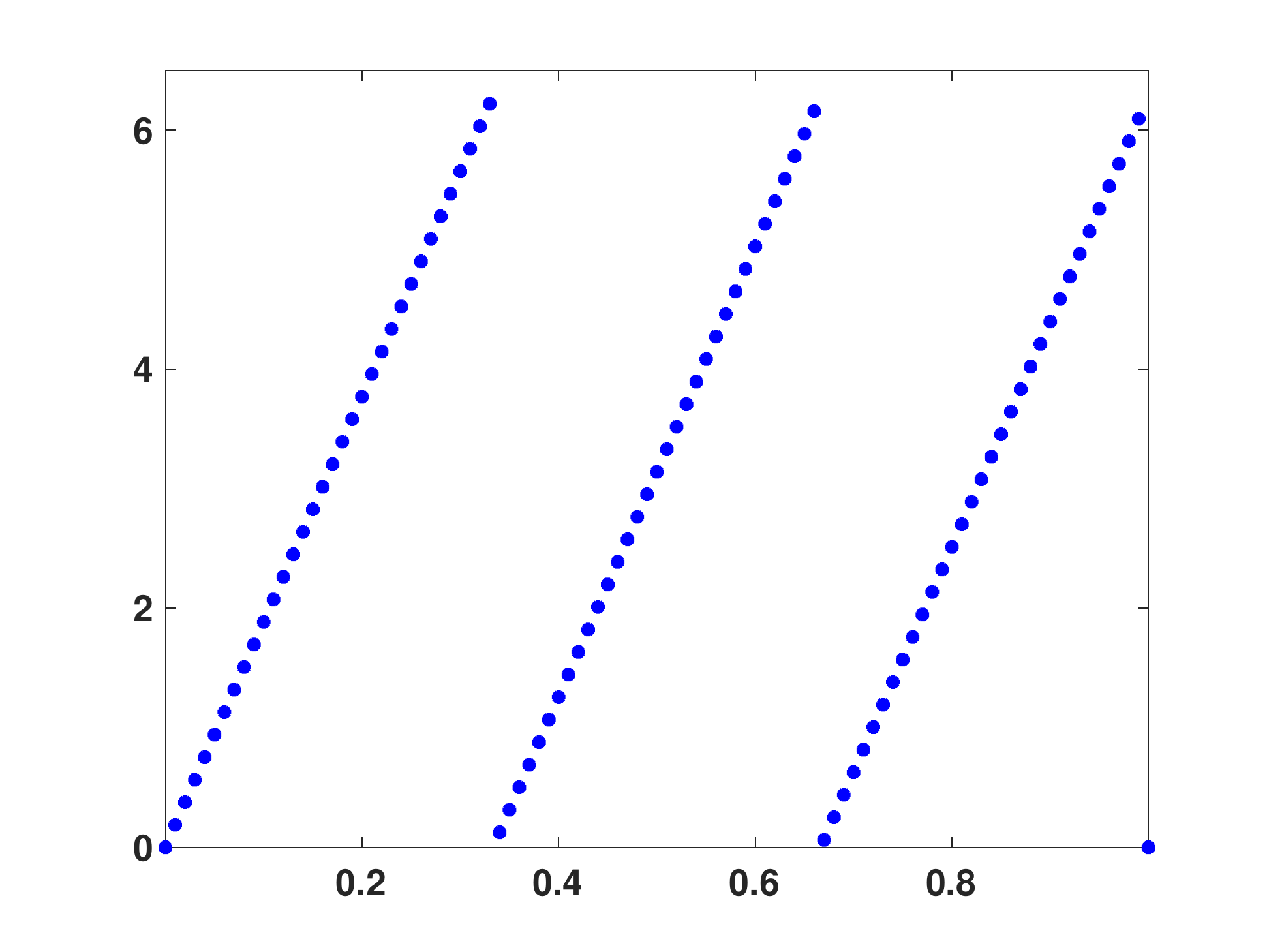}
\textbf{b}\includegraphics[height=1.5in, width=2.0in]{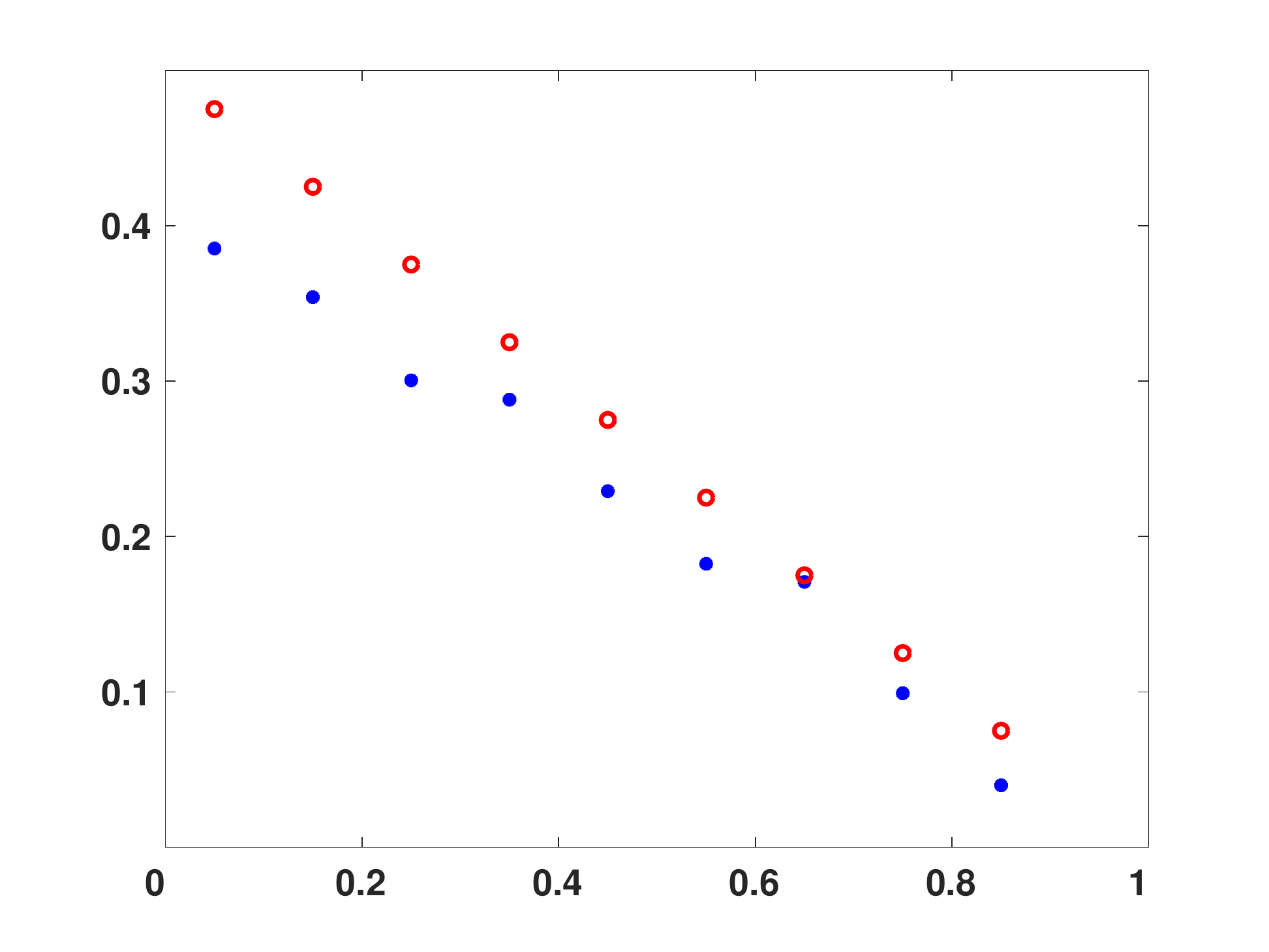}
% \textbf{b}\includegraphics[height=1.5in, width=2.0in]{./F/new-rate.pdf}
\end{center}
\caption{a) The 3-twisted state used to initialize the Kuramoto model \eqref{cKM}.
b) The numerically estimated exponent characterizing convergence of the numerical
scheme \eqref{KM}, \eqref{KM-ic}, $\alpha_\gamma$, is plotted as a function of $\gamma$
(see Section~\ref{sec.numerics}). The numerical estimates and the theoretical predictions are
plotted using of the blue stars and red circles respectively.
}
\lbl{f.rate}
\end{figure} 
\section{Numerical example}\lbl{sec.numerics}
In this section, we illustrate convergence analysis in the previous sections with a numerical example. 
To this end, we consider an IVP for the continuum Kuramoto model with nonlocal nearest--neighbor 
coupling \cite{MedWri17}:
\begin{eqnarray}\lbl{cKM}
\partial_t u(t,x) &=&\omega + \int_{[0,1]} K(y-x) \sin\left(u(t,x)-u(t,y)\right),\\
\lbl{cKM-ic}
u(0,x)&=& u^{(q)}(x),
\end{eqnarray}
where $u(t,x)\in \T,\; \T=\R/2\pi\Z,$ stands for the phase of the oscillator at $x\in [0,1]$,
$\omega\in \R$ is its intrinsic frequency. Function $K,$ describing  the connectivity of the 
network, is first defined on $[0,1/2)$ by
\be\lbl{def-K}
K(x)=\1_{\{y:\, |y|\le r\}}(x), \quad r\in (-1/2,1/2),
\ee
and then extended as a $1$--periodic function on $\R$. The initial condition
\be\lbl{q-twist}
u^{(q)}(x)=2\pi \left( qx \mod 1\right), \quad q\in\Z,
\ee
is called a $q$--twisted state (Figure~\ref{f.rate}a). For $\omega=0$, $u^{(q)}$ is a stationary solution
of \eqref{cKM}. Thus,  
\be\lbl{TW-solution}
u(t,x)= \left(2\pi qx +\omega t\right) \mod 2\pi
\ee
solves the IVP \eqref{cKM}, \eqref{cKM-ic}. We use the explicit solution \eqref{TW-solution}
to compute the error of the numerical integration of \eqref{cKM}, \eqref{cKM-ic}.

To estimate the rate of convergence of the numerical scheme \eqref{KM}, \eqref{KM-ic}
we use the following values of parameters: $r=0.2,$ $\omega=0.5,$ and $q=3.$ 
For these parameter values, the travelling wave solution \eqref{TW-solution} is unstable.
We integrated \eqref{cKM} numerically for $t\in[0,1],$ using the fourth order Runge--Kutta method
with the time step $10^{-2}.$ Note that the error of the Runge--Kutta method, i.e., of the discretization
in time is significantly 
smaller than of that of the discretizing in space (c.f.~\eqref{KM}, \eqref{KM-ic}).
We integrated \eqref{cKM} numerically for different values of $\gamma$ and for $n\in \{128, 256\}$.
For each pair $(\gamma, n)$ we repeated the numerical experiment $200$ times and computed 
the mean value of the error of numerical
integration (compared to the exact solution \eqref{TW-solution}). The mean errors $\bar e_{\gamma,128}$
and $\bar e_{\gamma, 256}$ computed for $n=128$ and $n=256$ respectively are used to determine
the convergence rate:
\be\lbl{compute_rate}
\alpha_\gamma={\ln \left(\bar e_{\gamma,128}/\bar e_{\gamma,256}\right)\over \ln 2}.
\ee
The results of this numerical experiment are shown in Figure~\ref{f.rate}b. Our main goal was to verity
the dependence of the convergence rate on sparsity controlled by $\gamma$. The pixel pictures for 
the adjacency matrices of random graphs corresponding to the nonlocal nearest--neighbor coupling
for $n=512$ and different values of $\gamma$ are shown in Figure~\ref{f.sparse}.
The plot in Figure~\ref{f.rate}b shows a clear linear relation between  the exponent $\alpha_\gamma$ 
and $\gamma$. The numerical rates plotted by blue stars are slightly lower the theoretical
rates $(1-\gamma)/2$ plotted in red. Overall numerical rates show a good fit with the analytical estimate.

\begin{figure}
\begin{center}
\textbf{a}\includegraphics[height=1.5in, width=2.0in]{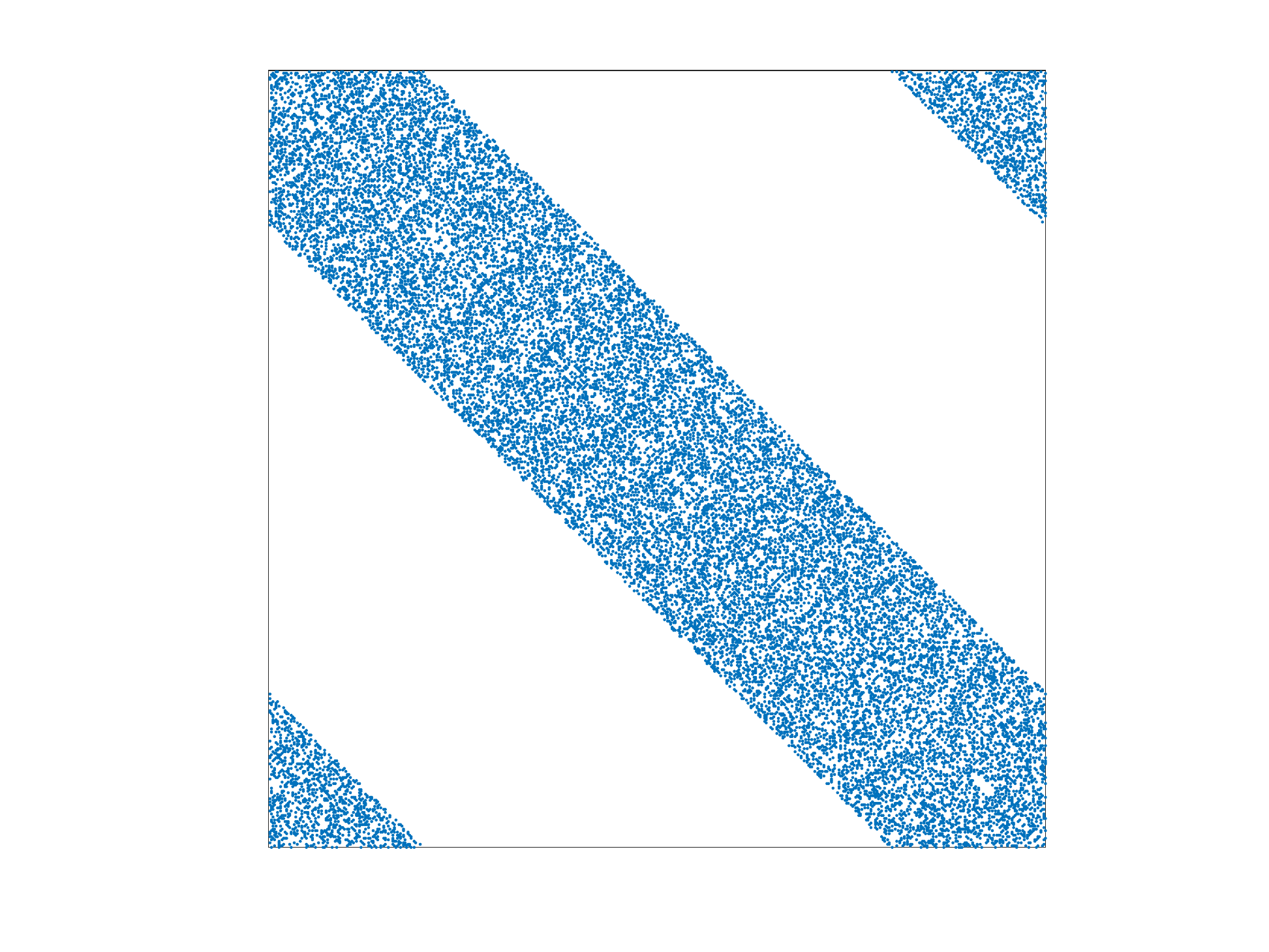} \quad
\textbf{b}\includegraphics[height=1.5in,width=2.0in]{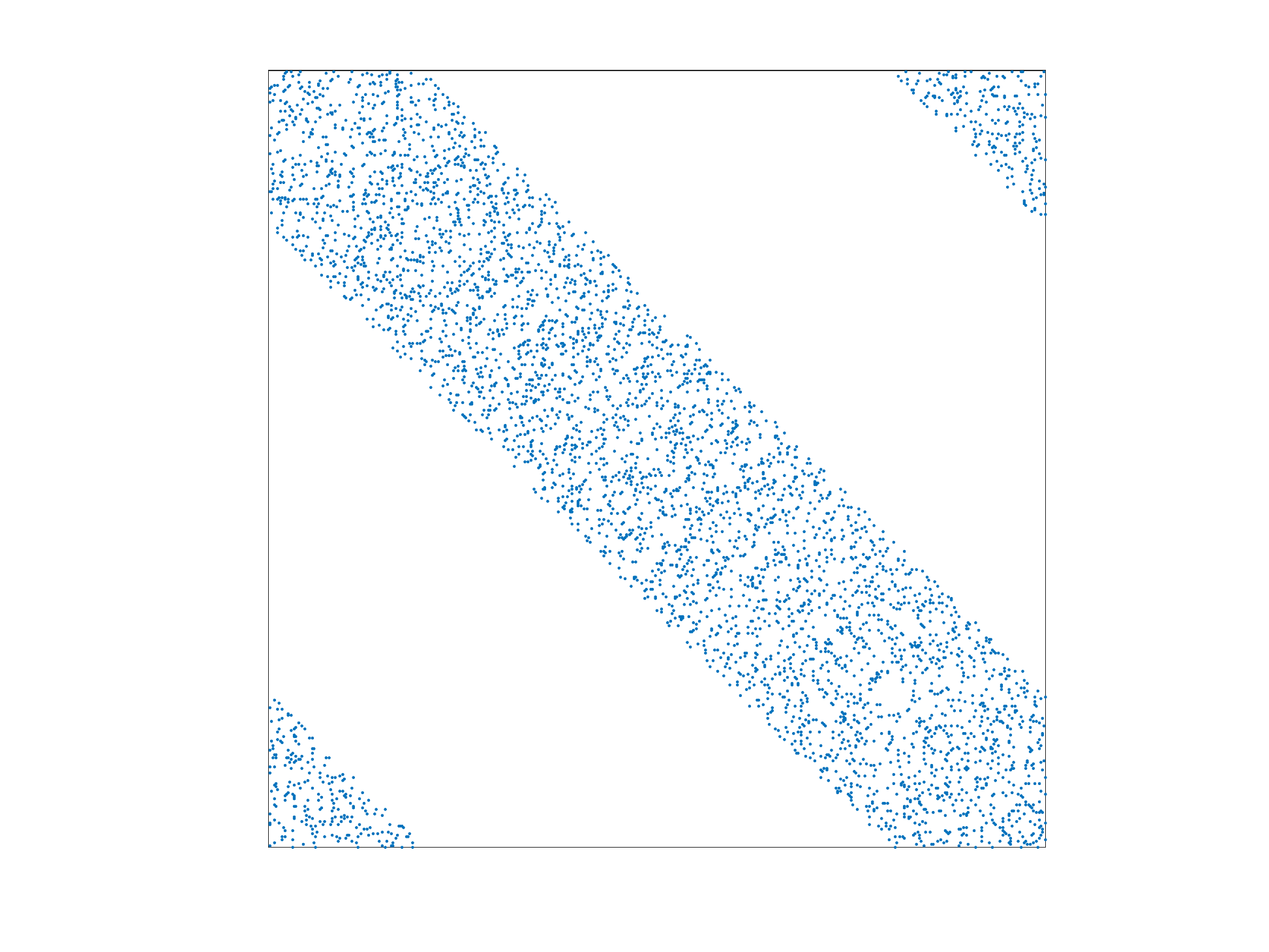} \\
\textbf{c}\includegraphics[height=1.5in, width=2.0in]{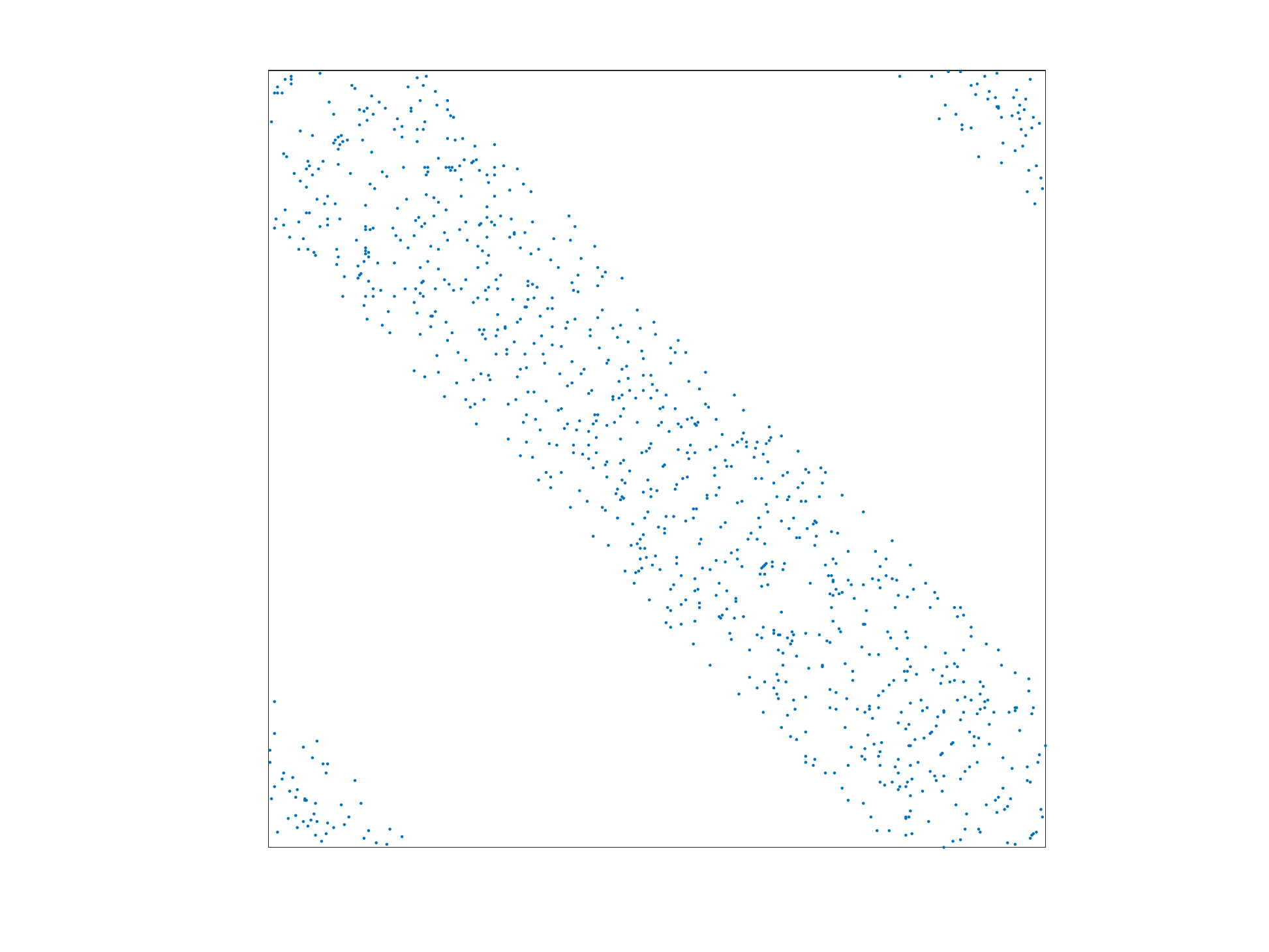} \quad
\textbf{d}\includegraphics[height=1.5in, width=2.0in]{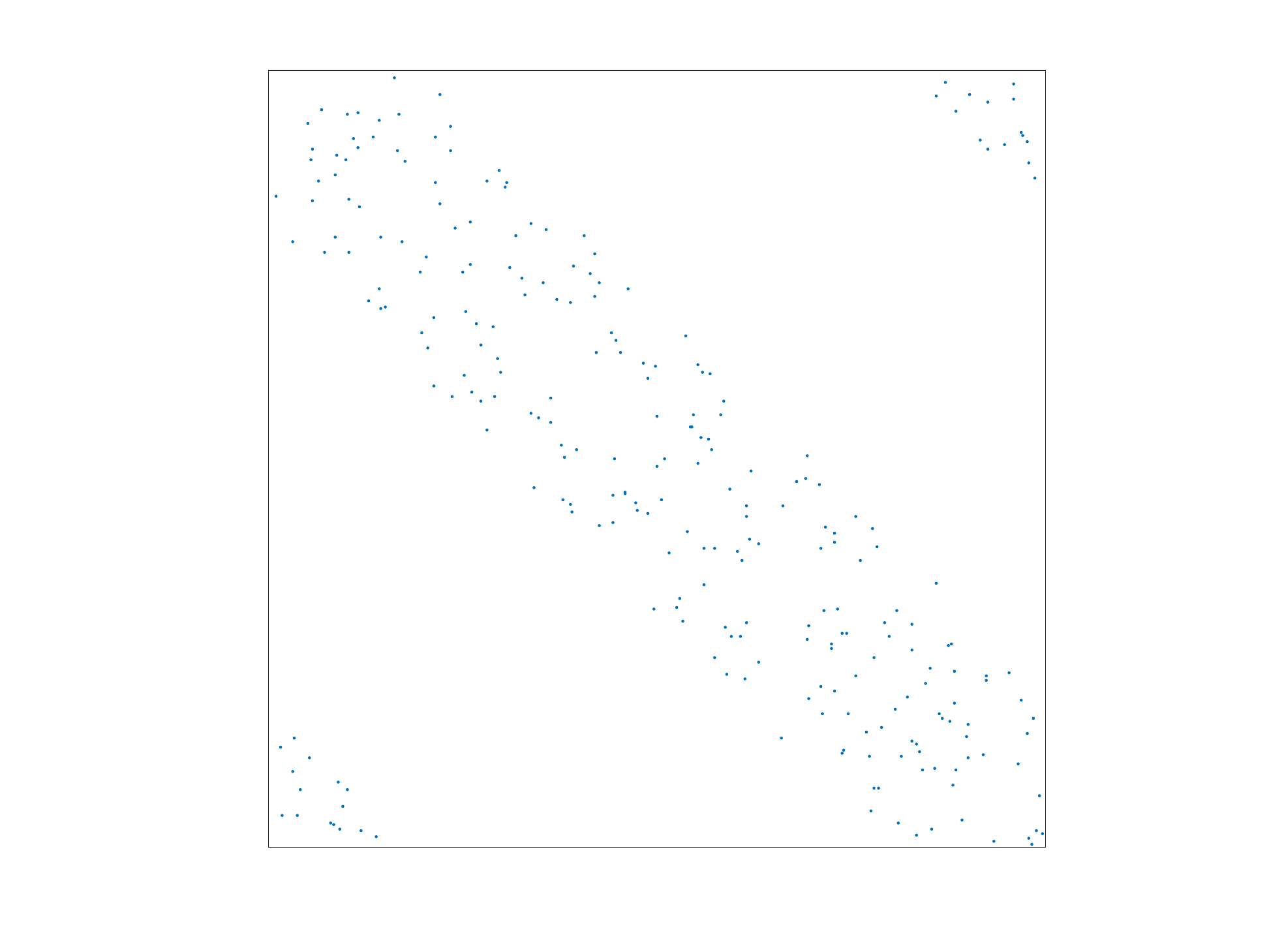}
\end{center}
\caption{Pixel pictures of the adjacency matrices of sparse graphs
  generated with the following values of  $\gamma$: a) $0.25$,
b) $0.5,$ c) $0.75,$ d) $0.95$.
}
\lbl{f.sparse}
\end{figure}

\section{Proof of Lemma~\ref{lem.ave}}\lbl{sec.average}
\setcounter{equation}{0}
In this section, we prove Lemma~\ref{lem.ave}. The proof follows the lines
of the proof of Theorem~4.1 in \cite{Med18}, which 
covers $\gamma\in (0,1/2)$ for $d=1$. Extension to the multidimensional case
$d>1$ is straightforward. Lemmas~\ref{lem.Bern} and \ref{lem.exp}
adapted from \cite{CDG18} allow to extend the range of $\gamma$ to $(0,1)$.
The reader not interested in the extended range of $\gamma$ may find
a simpler proof in \cite{Med18} easier to follow. For those interested in
the full range of $\gamma,$ below we present the following proof 
of Lemma~\ref{lem.ave}.

\begin{thm}\lbl{lem.stronger} %\cite[Corollary~4.1]{Med18}
Let nonnegative $W\in L^4(Q^2)$ satisfy
\be\lbl{W-1s}
\tag{W-1s}
\max\left\{  \esssup_{x\in Q} \int W^k (x,y) dy, \; \esssup_{y\in Q} \int W^k (x,y) dx
\right\} \le \bar W_k,\quad k\in[4]
\ee
and
\be\lbl{alpha}
\liminf_{n\to\infty} {\alpha_n n^d\over \ln n}>0.
\ee 
% $\alpha_n=N^{-\gamma},\; \gamma\in (0, 0.5)$ (cf.~\eqref{cutoff}).

Then for solutions of  \eqref{KM} and \eqref{aKM} subject to the 
same initial conditions and arbitrary $0<\epsilon<1/2$,
we have
\be\lbl{ave.stronger}
\sup_{t\in [0,T]} \| u_n(t,\cdot) -v_n(t,\cdot)\|_{L^2(Q)} \le 
C (\alpha_n n^d)^{1/2-\epsilon} \quad \mbox{a.s.},
\ee
for arbitrary $T>0$ and positive constant $C$
independent of $n$. 
% n^{-d(1/2-\gamma-\delta)}\quad \mbox{a.s.},
In particular, for $\alpha_n=n^{-d\gamma},\; \gamma\in (0, 1)$, we have
\be\lbl{ave.stronger}
\sup_{t\in [0,T]} \| u_n(t,\cdot) -v_n(t,\cdot)\|_{L^2(Q)} \le C
n^{-d(1-\gamma-\delta)/2}\quad \mbox{a.s.},
\ee
where $0<\delta<1-\gamma$ can be taken arbitrarily small.
\end{thm}

We precede the proof of Theorem~\ref{lem.stronger} with several auxiliary estimates. 
% \marginpar{address this:}
\begin{lem}\lbl{lem.simple}
From \eqref{W-1s} it follows 
\be\lbl{barW}
\max \left\{ \sup_{n\in\N} \max_{\bi\in\nd} n^{-d} \sum_{\bj\in\nd} W^k_{n,\bi\bj},
\;  \sup_{n\in\N} \max_{\bj\in\nd} n^{-d} \sum_{\bi\in\nd} W^k_{n,\bi\bj}
\right\}\le \bar W_k, \quad k\in [4].
\ee
\end{lem}
\begin{proof} We prove \eqref{barW} assuming
 that nonnegative $W$ is in $L^2(Q^2),$ but not in $L^\infty(Q^2)$. 
In this case, $W_{n,\bi\bj}$ are defined by \eqref{cutoff}.
For arbitrary $\bar k\in \nd$ and $n\in\N,$ we have 
\begin{equation}\lbl{discrete-bound}
\begin{split}
\sum_{\j\in\nd} W^k_{n,\bi,\bj} & = \sum_{ \bar j \in \nd} \left( n^d\int_{Q_{n,\bi}\times Q_{n,\bj}} \alpha_n^{-1}
\wedge W(x,y) dx dy \right)^k \\
& \le  \sum_{ \bar j \in \nd} n^d \int_{Q_{n,\bi}\times Q_{n,\bj}} \left(\alpha_n^{-k} \wedge W(x,y)^k \right) dx dy\\
& \le n^d \sum_{ \bar j \in \nd} \int_{Q_{n,\bi}\times Q_{n,\bj}}  W(x,y)^k  dx dy \\
&\le n^d \bar W_k,
\end{split}
\end{equation}
where we used Jensen's inequality in the second line and \eqref{W-1s} in the last line.
Thus,
$$
\sup_{n\in\N} \max_{\bi\in\nd} n^{-d} \sum_{\bj\in\nd} W^k_{n,\bi\bj} \le \bar W_k, \quad k\in [k].
$$
The bound for  $\sup_{n\in\N} \max_{\bj\in\nd} n^{-d} \sum_{\bi\in\nd} W^k_{n,\bi\bj}$ is proved similarly.
\end{proof}

\begin{lem}\lbl{lem.Bern} For $K\ge 2\bar W_1$, we have
\be\lbl{Bern}
\P\left( \max_{\bi\in\nd} \sum_{\bj \in \nd} \left| {a_{n,\bi\bj}\over \alpha_n} -W_{n,\bi\bj}\right| \ge 
Kn^d\right)
\le n^d \exp\left\{  
{ 
{-1\over 2} 
\left(K- 2\bar W_1 \right)^2 \alpha n^d 
\over
\bar W_1 +O(\alpha_n) + K } \right\}.
\ee
In particular, with probability $1$ there exists $n_0\in \N$ such that 
\be\lbl{Bern-implies}
\max\left\{
\max_{\bi\in \nd} \sum_{\bj\in\nd} \left| {a_{n,\bi\bj}\over \alpha_n}-W_{n,\bi\bj}\right|,
\max_{\bi\in \nd} \sum_{\bj\in\nd} \left| {a_{n,\bj\bi}\over \alpha_n}-W_{n,\bj\bi}\right|
\right\} \le Kn^d
\ee
for all $n\ge n_0$.
\end{lem}

For the next lemma, we will need the following notation
\begin{eqnarray}\lbl{Zn}
Z_{n,\bi}(t)   &=& n^{-d}\sum_{\bj\in\nd} b_{n,\bi\bj}(t)\eta_{n,\bi\bj},\\
\lbl{bn}
b_{n,\bi\bj}(t) &=& D\left(v_{n,\bj}(t)-v_{n,\bi}(t)\right),\\
\lbl{eta}
\eta_{n,\bi\bj} &=&a_{n,\bi\bj} -\alpha_n  W_{n,\bi\bj},
\end{eqnarray}
and $Z_n=(Z_{n,\bi},\; \bi\in\nd)$.
\begin{lem}\lbl{lem.exp} For arbitrary $\epsilon>0$, we have
\be\lbl{exp}
\alpha_{n}^{-2} \int_0^\infty e^{-Ls} \|Z_n (s)\|_{2,n^d}^2 ds \le C(\alpha_n n^d)^{1-\epsilon},
\ee
where $C$ is a positive constant independent of $n$ and
\be\lbl{dL2}
\|Z_n (s)\|_{2,n^d}=\left( n^{-1} \sum_{\bj\in\nd} Z_{n,\bj}(s)^2 \right)^{1/2}.
\ee
\end{lem}

\begin{proof}[Proof of Theorem~\ref{lem.stronger}]
Recall that $f(u,x,t)$ and $D$ are Lipschitz continuous function in $u$ with Lipschitz 
constants $L_f$ and $L_D$ respectively.
% In addition, $f(u,x,t)$ is a continuous function
% of $x$ and $D(u)$ is $2\pi$--periodic function satisfying \eqref{bound-D}.
 
Further,  $a_{n,ij},$ are Bernoulli random variables
\be\lbl{Bern}
\P(a_{n,\bi\bj}=1)=\alpha_n W_{n,\bi\bj}.
\ee

Denote $\psi_{n,\bi}:=v_{n,\bi}- u_{n,\bi}.$ 
By subtracting (\ref{KM}) from (\ref{aKM}), 
multiplying the result by $n^{-d}\psi_{n,\bi},$ and summing over $\bi\in\nd$, we obtain
\be \lbl{subtract}
\begin{split}
{1\over 2} {d\over dt} \|\psi_n\|^2_{2,n^d} &= \underbrace{
N^{-1}\sum_{\bi\in\nd} \left(f(v_{n,\bi},t)-f(u_{n,\bi},t)\right) \psi_{n,i}}_{I_1}+
\underbrace{ n^{-2d} \alpha_n^{-1}\sum_{\bi,\bj\in\nd}
\left(\alpha_n W_{n,\bi\bj}-a_{n,\bi\bj}\right)
D(v_{n,\bj}-v_{n,\bi}) \psi_{n,\bi}}_{I_2}\\
&+
\underbrace{N^{-2} \alpha_n^{-1} \sum_{\bi,\bj\in \nd}^n a_{n,\bi\bj} 
\left[ D(v_{n,\bj}-v_{n,\bi}) - D(u_{n,\bj}-u_{n,\bi})\right]\psi_{n,\bi}}_{I_3}=:I_1+I_2+I_3,
\end{split}
\ee
where $\| \cdot\|^2_{2,n^d}$ is the discrete $L^2$-norm (cf.~\eqref{dL2}).

Using Lipschitz continuity of $f$ in $u$, we  have
\be\lbl{I_1}
|I_1|\le L_f \|\psi_{n}\|_{2,n^d}^2.
\ee
Using Lipschitz continuity of $D$ and the triangle inequality, we 
have
\be\lbl{I_3}
\begin{split}
|I_3|& \le L_D n^{-2d} \alpha_n^{-1} \sum_{\bi,\bj\in \nd} a_{n,\bi\bj} 
\left( |\psi_{n,\bi}|+|\psi_{n,\bj}|\right) |\psi_{n,\bi}|\\
& \le L_D n^{-2d} \alpha_n^{-1} 
\left({3\over 2} \sum_{\bi,\bj\in\nd}  a_{n,\bi,\bj} \psi_{n,\bi}^2 +
{1\over 2} \sum_{\bi,\bj\in\nd}  a_{n,\bi\bj} \psi_{n,\bj}^2
\right).
\end{split}
\ee

Using Lemma~\ref{lem.Bern} and \eqref{barW}, we obtain
\be\lbl{uh-1}
\begin{split}
\alpha_n n^{-2d} \sum_{\bi,\bj\in\nd} a_{n,\bi\bj} \psi_{n,\bi}^2 & \le
n^{-d}\sum_{\bi\in\nd} 
\left\{ n^{-d}\sum_{\bj\in \nd} \left(\left| {a_{n,\bi\bj}\over \alpha_n} -W_{n,\bi\bj}\right|
+W_{n,\bi\bj}\right) \psi_{n,\bi}^2 \right\}\\
&\le n^{-d}\sum_{\bi\in\nd} \left(K+\bar W_1\right)\psi_{n,\bi}^2=  
\left(K+W_1\right) \|\psi_n\|_{2,n^d}^2.
\end{split}
\ee
Similarly,
\be\lbl{uh-2}
n^{-2d} \alpha_n^{-1} \sum_{\bi,\bj\in\nd} a_{n,\bi\bj\in\nd} \psi_{n,\bj}^2\le
\left(K+\bar W_2\right)\|\psi_n\|_{2,n^d}^2.
\ee

By plugging \eqref{uh-1} and \eqref{uh-2} into \eqref{I_3}, we have
\be\lbl{finish-I_3}
|I_3| \le L_D \left( 2K +{3\over 2} \bar W_1 +{1\over 2} \bar W_2 \right)
\|\psi\|_{n^d,2}^2.
\ee

It remains to bound $I_2$:
\be\lbl{est-I-1}
|I_2|= |n^{-d}\alpha_n^{-1}
\sum_{\bi\in\nd}^n Z_{n,\bi}\psi_{n,\bi}|\le 2^{-1}\alpha_n^{-2} \|Z_n\|_{2,n^d}^2 +
2^{-1} \|\psi_n\|^2_{2,n^d}.
\ee

The combination of  \eqref{subtract}, \eqref{I_1}, \eqref{finish-I_3} and \eqref{est-I-1} yields
\be\lbl{pre-G-1}
{d\over dt} \|\psi_n(t)\|_{2,n}^2 \le L\|\psi_n(t)\|_{2,n}^2 +
{1\over \alpha_n^2} \|Z_n(t)\|_{2,n}^2,
\ee
where $L=L_f+L_D\left(2K +{3\over 2} \bar W_1+{1\over 2} \bar W_2\right)+{1\over 2}.$

Using the Gronwall's inequality and Lemma~\ref{lem.exp}, we have
\be\lbl{Gron}
\begin{split}
\sup_{t\in [0, T]} \|\psi_n(t)\|_{2,n^d}^2 & \le \alpha_n^{-2} e^{ L T}
\int_0^\infty e^{-Ls}\|Z_n(s)\|_{2,n^d}^2 ds\\
&
\le \alpha_n^{-2} e^{ L T} (n^d\alpha_n)^{-1+\epsilon}.
\end{split}
\ee
\end{proof}

\begin{proof}[Proof of Lemma~\ref{lem.Bern}]
Let 
\be\lbl{xi}
\xi_{n,\bi\bj} =\left| {a_{n,\bi\bj}\over \alpha_n}-W_{n,\bi\bj} \right|-
2 W_{n,\bi\bj} \left(1-\alpha_n W_{n,\bi \bj} \right), \; \bi,\bj\in \nd.
\ee
Note that for fixed $\bi\in\nd,$ $\{\xi_{n,\bi\bj}, \, \bj\in \nd\}$ are mean zero independent RVs. Further,
using the definition of $\xi_{n,\bi\bj},$ it is straightforward to bound 
\begin{eqnarray}
\lbl{bound-M}
|\xi_{n,\bi\bj} | &\le& \alpha_n^{-1}+2 W_{n,\bi\bj}\le 3\alpha^{-1}_n=:M,\\
\lbl{bound-Exi2}
\E \xi_{n,\bj}^2 &\le & 2 \alpha_n^{-1} W_{n,\bi\bj} + 2W_{n,\bi\bj}^2 +4 \alpha_n  
W_{n,\bi\bj}^2+4 \alpha_n  W_{n,\bi\bj}^3.
\end{eqnarray}

From \eqref{bound-Exi2}, we have
\be\lbl{sum-xi2}
\begin{split}
\E\left(\sum_{\bj\in\nd} \xi_{n,\bi\bar j}^2\right) & \le 
\alpha_n^{-1} \sum_{\bar j\in\nd} \left( 
2 W_{n,\bi\bj} + \alpha_n 2W_{n,\bi\bj}^2 +4 \alpha^2_n  W_{n,\bi\bj}^2+4 \alpha^2_n  W_{n,\bi\bj}^3
\right), \\
 & \le \alpha_n^{-1} n^d W_1 +O\left(\alpha_n\right).
\end{split}
\ee

Using Bernstein's inequality and the union bound, we have
\be\lbl{Bernst}
\begin{split}
\P\left(\max_{\bi\in\nd}\sum_{\bj\in\nd} \xi_{n,\bi\bj} \ge \left(K-2\bar W_1\right)n^d 
\right) & \le 
n^d \exp\left\{ 
 {
       {-1\over 2} \left( K- 2\bar W_1 \right)^2 n^{2d}
\over
\sum_{\bj\in} \E\xi^2_{n,\bi\bj} + (1/3)M \left(K-2\bar W_1\right)n^d 
} 
\right\}\\
&\le 
n^d \exp\left\{ 
 {
       {-1\over 2} \left( K- 2\bar W_1 \right)^2 n^{2d}
\over
\alpha_n^{-1} n^d\left(\bar W_1+O(\alpha_n)\right) + \alpha_n^{-1} \left(K-2\bar W_1\right)n^d 
} 
\right\}\\
&\le 
N \exp\left\{  
{ 
{-1\over 2} 
\left(K-2\bar W_1 \right)^2 \alpha_n n^d 
\over
\bar W_1 +O(\alpha_n) + K 
} \right\}.
\end{split}
\ee

Finally, the combination of \eqref{xi} and \eqref{Bernst} yields
\begin{equation*}
\begin{split}
\P\left(\max_{\ni\in\nd}\sum_{\bj\in\nd} 
\left| {a_{n,\bi\bj}\over \alpha_n} -W_{n,\bi\bj}\right|
\ge Kn^d\right)
& \le 
\P\left(\max_{\ni\in\nd}\sum_{\bj\in\nd} \xi_{n,\bi\bj} \ge \left(K-{2\over n^d}
\sum_{\bj\in\nd} W_{n,\bi\bj}\right)n^d 
\right)\\
& \le 
\P\left(\max_{\ni\in\nd}\sum_{\bj\in\nd} \xi_{n,\bi\bj} \ge \left(K-2\bar W_1\right)n^d 
\right)\\
& 
\le 
n^d \exp\left\{  
{ 
{-1\over 2} 
\left(K- 2\bar W_1 \right)^2 \alpha n^d 
\over
\bar W_1 +O(\alpha_n) +K 
} \right\}.
\end{split}
\end{equation*}
This proves \eqref{Bern}. By Borel-Cantelli Lemma, \eqref{Bern-implies} follows.
\end{proof}

\begin{proof}[Proof of Lemma~\ref{lem.exp}]
Recall \eqref{Zn}-\eqref{eta} and rewrite
\be\lbl{Zni2-2}
\int_0^\infty e^{-Ls}  \|Z_{n}(s)\|_{2,n^d}^2 ds= 
n^{-3d} \sum_{\bi,\bk,\bl\in\nd} c_{n,\bi\bk\bl} \eta_{n,\bi\bk}\eta_{n,\bi\bl},
\ee
where
\be\lbl{def-cnjikl}
c_{n,\bi\bk\bl}=\int_0^\infty e^{-Ls} b_{n,\bi\bk}(s) b_{n\bi\bl}(s) ds \quad\mbox{and}\quad 
|c_{n,\bi\bk\bl}|\le L^{-1}=:\bar c.
\ee

By \eqref{alpha}, one can choose a sequence $\delta_n\searrow 0$ such that
\be\lbl{deltaN}
n^d\delta_n \gg \alpha_n^{-1}.
\ee
Specifically, let
\be\lbl{def-delta}
\delta_n:= {1\over \sqrt{\ln n}}.
\ee
and define events
\be\lbl{OmegaN}
\Omega_n= \left\{ (n^d\alpha_n)^{-2} \sum_{\bi,\b k,\b l\in\nd}^n c_{n,\bi\bk\bl} 
\eta_{n,\bi\bk}\eta_{n,\bi\bl} 
>\delta_n n^d   \right\},
\ee
\be\lbl{Ani}
A_{n,\bi}=\left\{ \sum_{\bj\in\nd} 
\left| {a_{n,\bi\bj} \over \alpha_n} -W_{n,\bi\bj}
\right|
>Kn^d\right\},\;\mbox{and}\; A_n=\bigcup_{\bi \in \nd} A_{n,\bi}.
\ee

Clearly,
\be\lbl{clear-bound}
\P( \Omega_n)\le \P(\Omega_n\cap A_n^c) +\P(A_n).
\ee
We want to show that $\P\left(\Omega_n\;\mbox{infinitely often}\right)=0$. By Borel-Cantelli Lemma,
it is sufficient to show that 
$$
\sum_{n\ge 1} \P( \Omega_n)<\infty.
$$
From Lemma~\ref{lem.Bern}, we know that 
$
\sum_{n\ge 1} \P( A_n)<\infty
$
for $K> 2\bar W_1.$ In the remainder of the proof,  we show that 
$\sum_{n\ge 1} \P(\Omega_n\cap A_n^c)$ is convergent.

Applying the exponential Markov inequality to $\P( \Omega_n\left| A_n^c\right. )$,  from 
$\P( \Omega_n\cap A_n^c )\le \P( \Omega_n\left| A_n^c\right. )$ and \eqref{OmegaN},  we have
\be\lbl{POmegan}
\P( \Omega_n\cap A_n^c) \le \exp\left\{ 
-n^d\delta_n  +\ln \E \left[ \1_{A_n^c} \exp\left\{
(n^d\alpha_n)^{-2} 
\sum_{\bi,\b k,\b l\in\nd}^n 
c_{n,\bi\bk\bl} \eta_{n,\bi\bk}\eta_{n,\bi\bl}
\right\} \right]
\right\}.
\ee

Using the independence of $\eta_{n,\bi\bk\bl}$ in $\bi\in \nd$, we have
\be\lbl{ind-in-i}
\E \left[ \1_{A_n^c} \exp\left\{
(n^d\alpha_n)^{-2} 
\sum_{\bi,\b k,\b l\in\nd} 
c_{n,\bi\bk\bl} \eta_{n,\bi\bk}\eta_{n,\bi\bl}
\right\} \right]=\prod_{\bi\in\nd}
\E \left[ \1_{A_n^c} \exp\left\{
(n^d\alpha_n)^{-2} 
\sum_{\b k,\b l\in\nd} 
c_{n,\bi\bk\bl} \eta_{n,\bi\bk}\eta_{n,\bi\bl}.
\right\} \right]
\ee

Using 
$$
e^x\le 1+|x|e^{|x|}, \quad x\in\R,
$$ 
and the Cauchy-Schwartz inequality, we bound the right--hand side of \eqref{ind-in-i}
as follows
\be\lbl{C-S}
\begin{split}
&\E \left[ \1_{A_n^c} \exp\left\{
(n^d\alpha_n)^{-2} 
\sum_{\b k,\b l\in\nd} 
c_{n,\bi\bk\bl} \eta_{n,\bi\bk}\eta_{n,\bi\bl}
\right\} \right] \\
& \le 1+ \E \left[ \1_{A_n^c}
\left|(n^d\alpha_n)^{-2} 
\sum_{\b k,\b l\in\nd} 
c_{n,\bi\bk\bl} \eta_{n,\bi\bk}\eta_{n,\bi\bl}
\right| \exp\left\{ \left|
(n^d\alpha_n)^{-2} 
\sum_{\b k,\b l\in\nd} 
c_{n,\bi\bk\bl} \eta_{n,\bi\bk}\eta_{n,\bi\bl}
\right|
\right\}
\right]\\
&\le 1+\left( \E \left\{(n^d\alpha_n)^{-2} 
\sum_{\b k,\b l\in\nd} 
c_{n,\bi\bk\bl} \eta_{n,\bi\bk}\eta_{n,\bi\bl} \right\}^2 \right)^{1/2}\\
&\times
\left(\E \left\{  \1_{A_n^c} \exp \left\{2(n^d\alpha_n)^{-2} 
\sum_{\b k,\b l\in\nd} 
c_{n,\bi\bk\bl} \eta_{n,\bi\bk}\eta_{n,\bi\bl} \right\}\right\}
\right)^{1/2}.
\end{split}
\ee

From \eqref{eta}, \eqref{def-cnjikl}, and under $A_n^c$ (cf.~\eqref{Ani}), 
we have 
\be\lbl{underAc}
 \1_{A_n^c} 2  (n^d\alpha_n)^{-2} \left|
\sum_{\b k,\b l\in\nd}^n 
c_{n,\bi\bk\bl} \eta_{n,\bi\bk}\eta_{n,\bi\bl} \right| \le 2K\bar c.
\ee
Further,
\be\lbl{moments}
\begin{split}
& \E \left\{(n^d\alpha_n)^{-2} 
\sum_{\b k,\b l\in\nd}
c_{n,\bi\bk\bl} \eta_{n,\bi\bk}\eta_{n,\bi\bl} \right\}^2 \le 
(n^d\alpha_n)^{-4} \sum_{\bj,\bk,\bl,\bar p\in \nd}
\E\left(\eta_{n,\bi\bj}\eta_{n,\bi\bk}\eta_{n,\bi\bl}\eta_{n,\bi\bar p}\right) c_{n,\bi\bj\bar p}
c_{n,\bi\bk\bl} \\
&\le  {(\bar c)^2 \over (n^d\alpha_n)^4}\left\{ \sum_{\bj\in\nd} \E \eta_{n, \bi\bj}^4 +
6\left(\sum_{\bj\in\nd}  \E\eta_{\bi\bj}^2\right)^2\right\}.
\end{split}
\ee

Using \eqref{eta}, we estimate sum of the fourth moments of $\eta_{n,\bi\bj}$
\be\lbl{moments-4}
\begin{split}
\sum_{\bj \in \nd} \E\eta^4_{\bi\bj} & =
\sum_{\bj \in \nd}  \left\{\alpha_nW_{n,\bi\bj} \left( 1-\alpha_n W_{n,\bi\bj} \right)^4
+ \alpha_n^4 W^4_{n,\bi\bj} \left( 1-\alpha_n W_{n,\bi\bj} \right)\right\}\\
& \le n^d\alpha_n \left( N^{-1} \sum_{\bj \in \nd} W_{n,\bi\bj}+
\alpha_n^3 n^{-d} \sum_{\bj \in \nd} W_{n,\bi\bj}^4\right)\\
&\le n^d\alpha_n \left( \bar W_1 + \alpha_n^3 \bar W_4\right) =O(\alpha_n n^d),
\end{split}
\ee
where we also use \eqref{barW}.
Similarly,
\be\lbl{moments-2}
\begin{split}
\sum_{\bj \in \nd} \E\eta^2_{\bi\bj} & =
\sum_{\bj \in \nd}  \left\{\alpha_nW_{n,\bi\bj} \left( 1-\alpha_n W_{n,\bi\bj} \right)^2
+ \alpha_n^2 W^2_{n,\bi\bj} \left( 1-\alpha_n W_{n,\bi\bj} \right)\right\}\\
& \le n^d\alpha_n \left( N^{-1} \sum_{\bj \in \nd} W_{n,\bi\bj}+
\alpha_n n^{-d} \sum_{\bj \in \nd} W_{n,\bi\bj}^2\right)\\
&\le n^d\alpha_n \left( \bar W_1 + \alpha_n \bar W_2\right) =O(\alpha_n N).
\end{split}
\ee

By combining \eqref{moments}-\eqref{moments-2}, we obtain
\be\lbl{cont-moments}
\E \left\{(n^d\alpha_n)^{-2} 
\sum_{\b k,\b l\in\nd}^n 
c_{n,\bi\bk\bl} \eta_{n,\bi\bk}\eta_{n,\bi\bl} \right\}^2 =O\left((n^d\alpha_n)^{-2}\right).
\ee

By plugging \eqref{underAc} and \eqref{cont-moments} into \eqref{C-S},
we obtain
\be\lbl{term}
\E \left[ \1_{A_n^c} \exp\left\{
(n^d\alpha_n)^{-2} 
\sum_{\b k,\b l\in\nd}^n 
c_{n,\bi\bk\bl} \eta_{n,\bi\bk}\eta_{n,\bi\bl}
\right\} \right]
\le 1 +{C_1\over n^d\alpha_n} e^{C_2}.
\ee
Using this bound on the right--hand side of \eqref{ind-in-i},
we further obtain
\be\lbl{back-to-product}
\E \left[ \1_{A_n^c} \exp\left\{
(n^d\alpha_n)^{-2} 
\sum_{\bi,\b k,\b l\in\nd}^n 
c_{n,\bi\bk\bl} \eta_{n,\bi\bk}\eta_{n,\bi\bl}
\right\} \right]\le e^{C_3\alpha_n^{-1}}.
\ee

Using \eqref{back-to-product}, from \eqref{POmegan} we obtain
\be\lbl{POmegan-back}
\P( \Omega_n\cap A_n^c )  \le \exp\left\{ 
-n^d\delta_{n^d} +  C_3\alpha_n^{-1} \right\}\to 0, \quad n\to\infty.
\ee
Furthermore, using \eqref{def-delta} it is straightforward to check that
$$
\sum_{n=1}^\infty \P( \Omega_n\cap A_n^c ) <\infty.
$$
The statement of the lemma then follows from \eqref{OmegaN}-\eqref{clear-bound}
via Borel-Cantelli Lemma.
\end{proof}

\vskip 0.2cm
\noindent
{\bf Acknowledgements.}
% The authors thank Dejan Slepcev for a useful discussion.
This work was supported in part by the NSF grant DMS 1715161.

\bibliographystyle{amsplain}
% \bibliography{redux}

\begin{thebibliography}{10}

\bibitem{Akh-Approx}
N.~I. Akhieser, \emph{Theory of {A}pproximation}, Dover Publications, Inc., New
  York, 1992, Translated from the Russian and with a preface by Charles J.
  Hyman, Reprint of the 1956 English translation. 

\bibitem{AMRT10}
Fuensanta Andreu-Vaillo, Jos\'{e}~M. Maz\'{o}n, Julio~D. Rossi, and
  J.~Juli\'{a}n Toledo-Melero, \emph{Nonlocal {D}iffusion {P}roblems}, Mathematical
  Surveys and Monographs, vol. 165, American Mathematical Society, Providence,
  RI; Real Sociedad Matem\'{a}tica Espa\~{n}ola, Madrid, 2010. 

\bibitem{Noch19}
Lehel Banjai, Jens~M. Melenk, Ricardo~H. Nochetto, Enrique Ot\'{a}rola,
  Abner~J. Salgado, and Christoph Schwab, \emph{Tensor {FEM} for spectral
  fractional diffusion}, Found. Comput. Math. \textbf{19} (2019), no.~4,
  901--962. 

\bibitem{BBPN18}
Andrea Bonito, Juan~Pablo Borthagaray, Ricardo~H. Nochetto, Enrique
  Ot\'{a}rola, and Abner~J. Salgado, \emph{Numerical methods for fractional
  diffusion}, Comput. Vis. Sci. \textbf{19} (2018), no.~5-6, 19--46.
 

\bibitem{BCCZ19}
Christian Borgs, Jennifer~T. Chayes, Henry Cohn, and Yufei Zhao, \emph{An
  {$L^p$} theory of sparse graph convergence {I}: {L}imits, sparse random graph
  models, and power law distributions}, Trans. Amer. Math. Soc. \textbf{372}
  (2019), no.~5, 3019--3062. 

\bibitem{BouCal10}
Nikolaos Bournaveas and Vincent Calvez, \emph{The one-dimensional
  {K}eller-{S}egel model with fractional diffusion of cells}, Nonlinearity
  \textbf{23} (2010), no.~4, 923--935. 

\bibitem{CarFife05}
C.~Carrillo and P.~Fife, \emph{Spatial effects in discrete generation
  population models}, J. Math. Biol. \textbf{50} (2005), no.~2, 161--188.
  

\bibitem{CGP14}
Stephen Coombes, Peter beim Graben, and Roland Potthast, \emph{Tutorial on
  neural field theory}, Neural fields, Springer, Heidelberg, 2014, pp.~1--43.
  

\bibitem{CDG18}
Fabio {Coppini}, Helge {Dietert}, and Giambattista {Giacomin}, \emph{A {L}aw of
  {L}arge {N}umbers and {L}arge {D}eviations for interacting diffusions on
  {E}rdos-{R}enyi graphs}, arXiv e-prints (2018), arXiv:1807.10921.

\bibitem{PabQuiRod16}
Arturo de~Pablo, Fernando Quir\'{o}s, and Ana Rodr\'{i}guez, \emph{Nonlocal
  filtration equations with rough kernels}, Nonlinear Anal. \textbf{137}
  (2016), 402--425. 

\bibitem{TEJ17}
F\'{e}lix del Teso, J\o~rgen Endal, and Espen~R. Jakobsen, \emph{Uniqueness and
  properties of distributional solutions of nonlocal equations of porous medium
  type}, Adv. Math. \textbf{305} (2017), 78--143. 

\bibitem{DeVore-book}
Ronald~A. DeVore and George~G. Lorentz, \emph{Constructive approximation},
  Grundlehren der Mathematischen Wissenschaften [Fundamental Principles of
  Mathematical Sciences], vol. 303, Springer-Verlag, Berlin, 1993. 

\bibitem{Du2018}
Qiang Du, \emph{An invitation to nonlocal modeling, analysis, and
  computation}, Proc. Int. Cong. of Math. (2018), Rio de Janeiro, vol. 3,
 pp.~3523–-3552.

\bibitem{Du2019b}
Qiang Du, Lili Ju, and Jianfang Lu, \emph{A discontinuous {G}alerkin method for
  one-dimensional time-dependent nonlocal diffusion problems}, Math. Comp.
  \textbf{88} (2019), no.~315, 123--147. 

\bibitem{Evans-fine}
Lawrence~C. Evans and Ronald~F. Gariepy, \emph{Measure theory and fine
  properties of functions}, revised ed., Textbooks in Mathematics, CRC Press,
  Boca Raton, FL, 2015. 

\bibitem{Falc-FracGeometry}
Kenneth Falconer, \emph{Fractal geometry}, third ed., John Wiley \& Sons, Ltd.,
  Chichester, 2014, Mathematical foundations and applications. 

\bibitem{Gol16}
Fran{\c{c}}ois Golse, \emph{On the dynamics of large particle systems in the
  mean field limit}, Macroscopic and large scale phenomena: coarse graining,
  mean field limits and ergodicity, Lect. Notes Appl. Math. Mech., vol.~3,
  Springer, [Cham], 2016, pp.~1--144. 

\bibitem{KVMed17}
Dmitry Kaliuzhnyi-Verbovetskyi and Georgi~S. Medvedev, \emph{The semilinear
  heat equation on sparse random graphs}, SIAM J. Math. Anal. \textbf{49}
  (2017), no.~2, 1333--1355. 

\bibitem{LovGraphLim12}
L.~Lov{\'a}sz, \emph{Large networks and graph limits}, AMS, Providence, RI,
  2012.

\bibitem{LovSze06}
L.~Lov{\'a}sz and B.~Szegedy, \emph{Limits of dense graph sequences}, J.
  Combin. Theory Ser. B \textbf{96} (2006), no.~6, 933--957. 

\bibitem{Med14a}
Georgi~S. Medvedev, \emph{The nonlinear heat equation on dense graphs and graph
  limits}, SIAM J. Math. Anal. \textbf{46} (2014), no.~4, 2743--2766.
  
\bibitem{Med14b}
\bysame, \emph{The nonlinear heat equation on {W}-random graphs}, Arch. Ration.
  Mech. Anal. \textbf{212} (2014), no.~3, 781--803. 

\bibitem{Med14c}
\bysame, \emph{Small-world networks of {K}uramoto oscillators}, Phys. D
  \textbf{266} (2014), 13--22. 

\bibitem{Med18}
\bysame, \emph{The continuum limit of the {K}uramoto model on
  sparse random graphs}, Communications in Mathematical Sciences \textbf{17}
  (2019), no.~4, 883--898.

\bibitem{MedTan15b}
Georgi~S. Medvedev and Xuezhi Tang, \emph{Stability of twisted states in the
  {K}uramoto model on {C}ayley and random graphs}, J. Nonlinear Sci.
  \textbf{25} (2015), no.~6, 1169--1208. 

\bibitem{MedTan18}
\bysame, \emph{The {K}uramoto model on power law graphs: Synchronization and
  contrast states}, Journal of Nonlinear Science (2018).

\bibitem{MedWri17}
Georgi~S. Medvedev and J.~Douglas Wright, \emph{Stability of twisted states in
  the continuum {K}uramoto model}, SIAM J. Appl. Dyn. Syst. \textbf{16} (2017),
  no.~1, 188--203. 

\bibitem{Du2019a}
Ricardo~H. Nochetto, Enrique Ot\'{a}rola, and Abner~J. Salgado, \emph{A {PDE}
  approach to space-time fractional parabolic problems}, SIAM J. Numer. Anal.
  \textbf{54} (2016), no.~2, 848--873. 

\bibitem{Noch16}
\bysame, \emph{A {PDE} approach to space-time fractional parabolic problems},
  SIAM J. Numer. Anal. \textbf{54} (2016), no.~2, 848--873. 

\bibitem{SchSil16}
Russell~W. Schwab and Luis Silvestre, \emph{Regularity for parabolic
  integro-differential equations with very irregular kernels}, Anal. PDE
  \textbf{9} (2016), no.~3, 727--772. 

\bibitem{ShenXie15}
Wenxian Shen and Xiaoxia Xie, \emph{Approximations of random dispersal
  operators/equations by nonlocal dispersal operators/equations}, J.
  Differential Equations \textbf{259} (2015), no.~12, 7375--7405. 

\bibitem{Vazquez17}
Juan~Luis V\'{a}zquez, \emph{The mathematical theories of diffusion: nonlinear
  and fractional diffusion}, Nonlocal and nonlinear diffusions and
  interactions: new methods and directions, Lecture Notes in Math., vol. 2186,
  Springer, Cham, 2017, pp.~205--278. 

\bibitem{WilStr06}
D.A. Wiley, S.H. Strogatz, and M.~Girvan, \emph{The size of the sync basin},
  Chaos \textbf{16} (2006), no.~1, 015103, 8. 

\bibitem{Williams-Prob-Mart}
David Williams, \emph{Probability with {M}artingales}, Cambridge Mathematical
  Textbooks, Cambridge University Press, Cambridge, 1991. 
\end{thebibliography}

\def\cprime{$'$} \def\cprime{$'$} \def\cprime{$'$}
\providecommand{\bysame}{\leavevmode\hbox to3em{\hrulefill}\thinspace}
\providecommand{\MR}{\relax\ifhmode\unskip\space\fi MR }
% \MRhref is called by the amsart/book/proc definition of \MR.
\providecommand{\MRhref}[2]{%
  \href{http://www.ams.org/mathscinet-getitem?mr=#1}{#2}
}
\providecommand{\href}[2]{#2}

\end{document}